\documentclass[12pt,a4paper,oneside]{amsart}

\usepackage{amsthm}
\usepackage{amsmath}
\usepackage{amssymb}
\usepackage{amscd}
\usepackage[utf8]{inputenc}
\usepackage{typearea}
\usepackage{eufrak}
\usepackage{yfonts}
\usepackage{textcomp}
\usepackage{mathrsfs}
\usepackage{hyperref}
\usepackage[draft]{fixme} 
\usepackage{pdfsync}
\usepackage[active]{srcltx}
\usepackage{xypic}
\usepackage{verbatim}

\renewcommand{\bar}[1]{\overline{#1}}

\newcommand{\PisoG}{\text{PIso}(\Cat,\dmap)}

\newcommand{\ZZ}{\mathbb{Z}}
\newcommand{\NN}{\mathbb{N}}

\newcommand{\ideQ}{\mathbf{1}_Q}

\newcommand{\ideG}{\mathbf{1}_G}

\newcommand{\osh}[1][]{\sigma_{#1}}

\newcommand{\dom}{\operatorname{dom}}
\newcommand{\ran}{\operatorname{ran}}

\newcommand{\G}{\mathcal{G}}

\newcommand{\HG}{\mathcal{H}}

\newcommand{\Cat}{\Lambda}

\newcommand{\Obj}{\Lambda^\circ}

\newcommand{\Inv}{\mathcal{I}}

\newcommand{\Grpd}{\mathcal{G}}
\newcommand{\Krpd}{\mathcal{K}}
\newcommand{\Grpdu}{\mathcal{G}^{(0)}}
\newcommand{\Grpdc}{\mathcal{G}^{(2)}}
\newcommand{\Semi}{\mathcal{S}}

\newcommand{\Idem}[1][\Semi]{\mathcal{E}(#1)}
\newcommand{\Idemp}{\mathcal{E}}

\newcommand{\dmap}{\mathbf{d}}
\newcommand{\tmap}{\mathbf{t}}
\newcommand{\dbar}{\bar{\mathbf{d}}}

\newcommand{\elmap}[2]{
	\tau^{#1}\osh^{#2}}

\newtheorem{lemma}{Lemma}[section]
\newtheorem{corollary}[lemma]{Corollary}
\newtheorem{theorem}[lemma]{Theorem}
\newtheorem{proposition}[lemma]{Proposition}

\theoremstyle{definition}
\newtheorem{definition}[lemma]{Definition}

\newtheorem{example}[lemma]{Example}

\newtheorem{notation}[lemma]{Notation}
\newtheorem{remark}[lemma]{Remark}
\newtheorem{remas}[lemma]{Remarks}

\title[Groupoid partial actions]{Zappa-Sz\'ep products for partial actions of groupoids on Left Cancellative Small Categories.}

\begin{document}

\author{Eduard Ortega}
\address{Department of Mathematical Sciences\\
NTNU\\
NO-7491 Trondheim\\
Norway } \email{Eduardo.Ortega@math.ntnu.no}

\author{Enrique Pardo}
\address{Departamento de Matem\'aticas, Facultad de Ciencias\\ Universidad de C\'adiz, Campus de
Puerto Real\\ 11510 Puerto Real (C\'adiz)\\ Spain.}
\email{enrique.pardo@uca.es}\urladdr{https://sites.google.com/a/gm.uca.es/enrique-pardo-s-home-page/}


\thanks{The second-named author was partially supported by PAI III grant FQM-298 of the Junta de Andaluc\'{\i}a, by the European Union under the 2014-2020 ERDF Operational Programme and by the Department of Economy, Knowledge, Business and University of the Junta de Andaluc\'{\i}a, jointly, through grant FEDER-UCA18-107643, and by the DGI-MINECO and European Regional Development Fund, jointly, through grants MTM2017-83487-P and PID2020-113047GB-I00}



\date{\today}

 
 \begin{abstract}
We study groupoid actions on left cancellative small categories and their associated Zappa-Sz\'ep products. We show that certain left cancellative small categories with nice length functions can be seen as Zappa-Sz\'ep products. We compute the associated tight groupoids, characterizing important properties of them, like being Hausdorff, effective and minimal. Finally, we determine amenability of the tight groupoid under mild, reasonable hypotheses. 
 \end{abstract}
 
 \maketitle
 
\section*{Introduction} 

 In \cite{S1}, Spielberg described a new method of defining $C^*$-algebras associated to oriented combinatorial data, generalizing the construction of algebras from directed graphs, higher-rank graphs, and (quasi-)ordered groups. To this end, he introduced \emph{categories of paths} --i.e. cancellative small categories with no (nontrivial) inverses-- as a generalization of higher rank graphs, as well as ordered groups. The idea is to start with a suitable combinatorial object and define a $C^*$-algebra directly from what might be termed the generalized symbolic dynamics that it induces. Associated to the underlying symbolic dynamics, he presents a natural groupoid derived from this structure. The construction also gives rise to a presentation by generators and relations, tightly related to the groupoid presentation. In \cite{S2} he showed that most of the results hold when relaxing the conditions, so that right cancellation or having no (nontrivial) inverses are taken out of the picture. 
 
 In \cite{OP_LCSC}, the authors studied Spielberg's construction, using a groupoid approach based in the Exel's tight groupoid construction \cite{E1}, showing that the tight groupoid for these inverse semigroups coincide with Spielberg's groupoid \cite{S-LMS}. With this tool at hand, they were able to characterize simplicity for the algebras associated to finitely aligned left cancellative small categories, and in particular in the case of Exel-Pardo systems \cite{EP2}. Finally, they gave, under mild and necessary hypotheses, a characterization of amenability for such a groupoid.
 
  Therefore, it becomes important to understand the internal structure of the left cancellative small category to check the desired properties of the associated groupoid, and hence of its associated ($C^*$-)algebra. A classical idea is to decompose our complex object in different simple pieces with well-behaved relations between them. This was well studied in \cite{LW,LV2}, where it was proved that categories with length functions on $\NN^k$ with certain decomposition properties can be written as the Zappa-Sz\'ep product of the groupoid of invertible elements of the category and a higher-rank graph subcategory generated by a transversal of generators of maximal right ideals. Zappa-Sz\'ep products of left cancellative small categories and groups were studied by  B\'edos, Kaliszewski, Quigg and Spielberg in 
  \cite{BKQS}, where they studied the representation theory for the Spielberg algebras of the new left cancellative small category associated to this construction.
   
In the present paper, we extend the scope of  \cite{BKQS} to actions of groupoids, To this end, we define groupoid actions on a left cancellative small category and their Zappa-Sz\'ep products,
    and we show that Zappa-Sz\'ep products appear naturally  in the context of left cancellative small categories with  length functions. Finally, we will extend the results of \cite[Sections 7 \& 8]{OP_LCSC} to determine the essential properties of the tight groupoid associated to  Zappa-Sz\'ep products of groupoid actions on a left cancellative small category, including the amenability of its tight groupoid.
   
    The contents of this paper can be summarized as follows: In Section 1 we recall some known results on small categories, and we define length functions and factorization properties we will need in the sequel.  In Section 2 we define groupoid actions on left cancellative small categories. In Section 3 we define the Zappa-Sz\'ep products of certain groupoid actions on left cancellative small categories. In Section 4 we show that left cancellative small categories with nice length functions can be described as Zappa-Sz\'ep products of the action of their groupoid of invertible elements  on certain nice subcategories.  In Section 5 we analyze the structure of the tight groupoid associated to Zappa-Sz\'ep products of groupoid actions to left cancellative small categories. We close the paper studying, in Section 6, the amenability of the tight groupoid of this kind of Zappa-Sz\'ep products.

\section{Small categories.}

In this section we collect all the basic background material about small categories  we need for the rest of the paper. For more details see \cite{OP_LCSC}.

Given a small category $\Cat$, we will denote by $\Obj$ its objects, and we will identify $\Obj$ with the identity morphisms, so that $\Obj\subseteq \Cat$. Given $\alpha\in\Cat$, we will denote by $s(\alpha):=\dom(\alpha)\in \Obj$ and $r(\alpha):=\ran(\alpha)\in \Obj$. The right invertible elements of $\Cat$ are 
$$\Cat^{-1}:=\{\alpha\in\Cat: \exists \beta\in\Cat\text{ such that }\alpha\beta=s(\beta)\}\,.$$
\begin{definition}
	Given a small category $\Cat$, and let $\alpha,\beta,\gamma\in\Cat$:
	\begin{enumerate}
		\item $\Cat$ is \emph{left cancellative} if $\alpha\beta=\alpha\gamma$ then $\beta=\gamma$,
		\item $\Cat$ is \emph{right cancellative} if $\beta\alpha=\gamma\alpha$ then $\beta=\gamma$,
		\item $\Cat$ \emph{has no inverses} if $\alpha\beta=s(\beta)$ then $\alpha=\beta=s(\beta)$.
	\end{enumerate}

A \emph{category of paths} is a small category that is right and left cancellative and has no inverses. 
\end{definition}

Notice that if $\Cat$ is either left or right cancellative, then the only idempotents in $\Cat$ are $\Obj$. Therefore, given $\alpha\in \Cat^{-1}$ with right inverse $\beta$,  we have that $\beta\alpha=s(\alpha)$. Thus, $\Cat^{-1}$ is the set of the right and left invertible morphisms. Given $\alpha\in \Cat^{-1}$ we will denote by $\alpha^{-1}$ its inverse. 

\begin{definition}\label{defi1_1_1_2}
	Let $\Cat$ be a small category. Given $\alpha,\beta\in\Cat$, we say that $\beta$ \emph{extends $\alpha$} (equivalently \emph{$\alpha$ is an initial segments of $\beta$}) if there exists $\gamma\in\Lambda$ such that $\beta=\alpha\gamma$. We denote by $[\beta]=\{\alpha\in\Cat:\alpha\text{ is an initial segment of }\beta\}$. We write $\alpha\leq \beta$ if $\alpha\in[\beta]$.
\end{definition}

Let $\Cat$ be a left cancellative small category. Then given $\alpha,\beta\in \Cat$, then the following is equivalent 
\begin{enumerate}
	\item $\alpha\leq \beta$ and $\beta\leq \alpha$ ($\alpha\approx\beta$),
	\item $\beta\in \alpha\Cat^{-1}$,
	\item $\alpha\in \beta\Cat^{-1}$,
	\item $\alpha\Cat=\beta\Cat$,
	\item $[\alpha]=[\beta]$.
\end{enumerate}

\begin{notation}
	Let $\Cat$ be a left cancellative small category. Given $\alpha,\beta\in\Cat$, we say :
	\begin{enumerate}
		\item $\alpha\Cap \beta$ if and only if $\alpha\Cat\cap\beta \Cat\neq\emptyset$,
		\item $\alpha\perp \beta$ if and only if $\alpha\Cat\cap\beta \Cat=\emptyset$.
	\end{enumerate}
\end{notation}

\begin{definition}
	Let $\Cat$ be a left cancellative small category, and let $F\subset \Cat$. The elements of $\bigcap\limits_{\gamma\in F}\gamma\Cat$ are \emph{the common extensions of $F$}. A common extension $\varepsilon$ of $F$  is \emph{minimal} if for any common extension $\gamma$ with $\varepsilon\in \gamma\Cat$ we have that $\gamma\approx \varepsilon$.
\end{definition}

When $\Cat$ has no inverses, given $F\subseteq \Cat$ and given any minimal common extension $\varepsilon$ of $F$, if $\gamma$ is common extension of  $F$ with $\varepsilon\in \gamma\Cat$ then $\gamma= \varepsilon$. We will denote by 
$$\alpha\vee \beta:=\{\text{the minimal extensions of }\alpha\text{ and }\beta\}\,.$$
Notice that if $\alpha\vee \beta\neq \emptyset$ then $\alpha\Cap \beta$, but the converse fails in general. 

\begin{definition}
	A left cancellative small category $\Lambda$ is \emph{finitely aligned} if for every $\alpha,\beta\in\Lambda$ there exists a finite subset $\Gamma\subset \Lambda$ such that $\alpha\Lambda\cap\beta\Lambda=\bigcup\limits_{\gamma\in \Gamma}\gamma\Lambda$.
\end{definition}

When $\Lambda$ is a finitely aligned left cancellative small category, we can always assume that $\alpha\vee\beta=\Gamma$ where $\Gamma$ is a finite set of minimal common extensions of  $\alpha$ and $\beta$.
\begin{definition}\label{definition2_3_7}
	Let $\Lambda$ be a LCSC and $\alpha\in \Lambda$. A subset $F\subset r(\alpha)\Lambda$ is \emph{exhaustive with respect to $\alpha$} if for every $\gamma\in \alpha\Lambda$ there exists a $\beta\in F$ with $\beta\Cap\gamma$. We denote $\mathsf{FE}(\alpha)$ the collection of finite sets of $r(\alpha)\Lambda$ that are exhaustive with respect to $\alpha$.
\end{definition}

\begin{definition}
	Let $\Cat$ be a LCSC and let $\Gamma\subseteq  Q$ be a submonoid of a group $Q$ with unit element $\ideQ$, and such that  $\Gamma\cap\Gamma^{-1}=\{\ideQ\}$. A map $\dmap:\Lambda\to\Gamma$ is called  a \emph{length function} if
	\begin{enumerate}
		\item $\dmap(\alpha\beta)=\dmap(\alpha)\dmap(\beta)$ for every $\alpha,\beta\in\Lambda$ with $s(\alpha)=r(\beta)$,
\end{enumerate}

A length function is say to satisfy the \emph{weak factorization property} if
\begin{enumerate}	
	\item[\textbf{(WFP)}]  for every $\alpha\in \Lambda$ and $\gamma_1,\gamma_2\in \Gamma$ with $\dmap(\alpha)=\gamma_1\gamma_2$, there are  $\alpha_1,\alpha_2\in \Lambda$ with $s(\alpha_1)=r(\alpha_2)$, $\dmap(\alpha_i)=\gamma_i$ for $i=1,2$, such that $\alpha=\alpha_1\alpha_2$, and moreover for every $\beta_1,\beta_2\in \Lambda$ with $s(\beta_1)=r(\beta_2)$, $\dmap(\beta_i)=\gamma_i$ for $i=1,2$, such that $\alpha=\beta_1\beta_2$, there exists $g_1,g_2\in \Cat^{-1}$ such that $\beta_1=\alpha_1g_1$ and $\beta_2=g_2\alpha_2$. 
	
\end{enumerate}	 
\end{definition}
Observe that, given a LCSC $\Cat$, always exists what we will call the \emph{trivial length function} $\dmap:\Cat \to \Gamma$, defined by $\dmap(\alpha)=\ideQ$ for every $\alpha\in \Cat$. 

The above definition is motivated by the one given  in \cite[Section 3]{LV2} where there the authors only consider the case when $\Gamma=\NN^k$.

\begin{remas}\label{rema_WFP} 
	Let $\Cat$ be a LCSC  and let $\Gamma\subseteq Q$ be a submonoid of a group $Q$  with unit element $\ideQ$, and such that  $\Gamma\cap\Gamma^{-1}=\{\ideQ\}$. Then, given  a length function   $\dmap:\Lambda\to\Gamma$  satisfying the WFP, we have:
	\begin{enumerate}
		\item  $\dmap^{-1}(\ideQ)=\Cat^{-1}$. Indeed, first observe that, given $v\in \Cat^0$, we have that $\dmap(v)=\dmap(vv)=\dmap(v)\dmap(v)$, and hence $\dmap(v)=\ideQ$. Now, let $g\in \Cat^{-1}$. Then $$\ideQ=\dmap(s(g))=\dmap(g^{-1}g)=\dmap(g^{-1})\dmap(g)\,.$$
		Therefore, $\dmap(g^{-1})=\dmap(g)^{-1}$, whence $\dmap(g)=\ideQ$. Now, let $\alpha\in \Cat$ such that $\dmap(\alpha)=\ideQ$. Since $\alpha=\alpha s(\alpha)=r(\alpha)\alpha$, we have that $\dmap(\alpha s(\alpha))=\ideQ\ideQ=\dmap(r(\alpha)\alpha)$. By the WFP, there exists $g_1,g_2\in \Cat^{-1}$ such that $\alpha=r(\alpha)g_1$ and $s(\alpha)=g_2\alpha$. Thus, $\alpha=g_1\in \Cat^{-1}$.		\item 	If $\Cat$ has no inverses, then the weak factorization property is equivalent to the so-called \emph{unique factorization property} \cite[Defintion 6.1]{RW}. 
	\end{enumerate}

\end{remas}
 

\begin{definition}
	A LCSC $\Cat$ is called \emph{action-free} if the action of $\Lambda^{-1}$ on $\Lambda$ is free; that is, whenever $g\gamma=\gamma$ for some $\gamma\in \Cat$ and $g\in \Cat^{-1}$, then $g=r(\gamma)$.
\end{definition}

Observe that if $\Cat$ has no inverses, then $\Cat$ is action-free.

\begin{lemma}\label{pseudofree_cat_right_cancellative}
	Let $\Cat$ be a LCSC and let  $\dmap:\Cat\to \Gamma$ be a length function satisfying the WFP.  Then $\Cat$ is right cancellative if and only if $\Cat$ is action-free.
\end{lemma}
\begin{proof}
	Let $\alpha,\beta,\gamma\in \Cat$ with $\alpha\gamma=\beta\gamma$, there exist $g_1,g_2\in \Cat^{-1}$ with $\beta=\alpha g_1$ and $\gamma=g_2\gamma$. Therefore, $\alpha\gamma=\alpha g_1g_2\gamma$, so $\gamma=g_1g_2\gamma$ by left cancellation and $g_1g_2=r(\gamma)$ by action-freeness. Hence $g_2=g_1^{-1}$. Finally we have that $g_1^{-1}\gamma=\gamma$, and using action-freeness again we have that $g_1^{-1}=r(\gamma)$ and so $\beta=\alpha$.
\end{proof}

\begin{definition}
	A \emph{groupoid} is a category such that every morphism has an inverse. Given a  groupoid $\Grpd$ we denote by $\Grpdu$ the set of identity functions on its objects (\emph{units of $\Grpd$}). Moreover, we define the \emph{range} and \emph{source} functions $r:\Grpd\to \Grpdu$ and $s:\Grpd\to \Grpdu$ by $r(g)=gg^{-1}$ and $s(g)=g^{-1}g$ for every $g\in \Grpd$. We denote by $\Grpdc:=\{(g,h)\in \Grpd\times \Grpd: s(g)=r(h)\}$ the set of \emph{composable} morphisms. A \emph{topological groupoid} is a groupoid with a topology that makes the product and inverse operations continuous. A \emph{discrete groupoid} is a topological groupoid with the discrete topology.
\end{definition}

\begin{example}
	Given a LCSC $\Cat$, $\Cat^{-1}$ is a (discrete) groupoid where  we can identify $(\Cat^{-1})^{(0)}$  with $\Cat^0$. 
\end{example}

\section{Groupoid partial actions}

In this section, we will define actions of groupoids on left cancellative small categories.  This is inspired in the construction of partial actions of groupoids on graphs \cite[Sections 2 \& 3]{LRRW}.\vspace{.2truecm}

For the rest of the paper we will assume that  $\Gamma\subseteq Q$ is a submonoid of a group $Q$  with unit element $\ideQ$ such that  $\Gamma\cap\Gamma^{-1}=\{\ideQ\}$.

The first step is to define the notion of partial isomorphism of a small category $\Cat$ inspired in \cite[Definition 3.1]{LRRW}. 

\begin{definition}\label{Def: Partial_Isomorphism}
Let $\Cat$ be a LCSC  with length function $\dmap:\Lambda\to\Gamma$. Let $\PisoG$ be the set of partial isomorphisms of $\Cat$, $f\in \PisoG$ such that satisfies: 
\begin{enumerate}
\item $f:v\Cat \rightarrow w\Cat$ is a bijection  for some  $v,w\in \Cat^0$,
\item $f(\alpha\Cat)=f(\alpha)\Cat$ for every $\alpha\in v\Cat$.
\item $f(v)=w$,
\item $\dmap(f(\alpha))=\dmap(\alpha)$ for every $\alpha\in v\Cat$.
\end{enumerate}
We will denote by $v=:d(f)$ (the \emph{domain of $f$}) and  by $w=:c(f)$ (the \emph{codomain of $f$}).
\end{definition}

\begin{remark}\label{Rem:ConditionsAction}$\mbox{ }$\vspace{.1truecm}

\begin{enumerate}
	\item Observe that condition $(3)$ does not follow from the previous conditions. 
	Indeed, given $g\in \Cat^{-1}$, the map $f:s(g)\Cat\to r(g)\Cat$ given by $f(\gamma)=g\gamma$ for every $\gamma\in s(g)\Cat$ satisfies conditions $(1)-(2)$ but not $(3)$ whenever $g\in \Cat^{-1}\setminus \Cat^0$.
	\item Given $f:v\Cat\to w\Cat$ in $\PisoG$ and $g\in v\Cat^{-1}$ we have that $f(gg^{-1})=w$ by condition $(3)$, and hence by condition $(2)$ there exists $h\in \Cat$ such that $f(g)h=w$, so $f(g)\in \Cat^{-1}$.
\end{enumerate}
\end{remark}

\begin{lemma}\label{cocycle}
	Let $\Cat$ be a LCSC  with length function $\dmap:\Lambda\to\Gamma$. Then, given $f\in \PisoG$ and  $\alpha\in d(f)\Cat$, there exists a unique function $f_{|\alpha}:s(\alpha)\Cat\to s(f(\alpha))\Cat$ in $\PisoG$ such that $f(\alpha\beta)=f(\alpha)f_{|\alpha}(\beta)$ for every $\beta\in s(\alpha)\Cat$.
\end{lemma}
\begin{proof}
Given  $f\in \PisoG$ and $\alpha\in d(f)\Cat$  we define 
$$
\begin{array}{rrcl}
f_{|\alpha}: & s(\alpha)\Cat & \rightarrow  &  s(f(\alpha))\Cat \\
& \beta & \mapsto  &  \gamma 
\end{array}
$$
where $\gamma\in s(f(\alpha))\Lambda$ is such that $f(\alpha\beta)=f(\alpha)\gamma$. The existence of $\gamma$ is guaranteed by condition (2). Now suppose that there exist $\gamma_1,\gamma_2\in s(f(\alpha))\Cat$ such that $f(\alpha\beta)=f(\alpha)\gamma_1=f(\alpha)\gamma_1$. Then, by left cancellation, we have that $\gamma_1=\gamma_2$, so $f_{|\alpha}$ is well-defined. 

Now suppose there exist $\beta_1, \beta_2\in s(\alpha)\Lambda$ that satisfy $f(\alpha)\gamma=f(\alpha\beta_1)=f(\alpha\beta_2)$ for some $\gamma\in s(f(\alpha))\Cat$. Since $f$ is bijective, $\alpha\beta_1=\alpha\beta_2$, whence $\beta_1=\beta_2$ by left cancellation. Thus, $f_{|\alpha}$ is injective.   Finally, if $\gamma\in s(f(\alpha))\Cat$, then $f(\alpha)\gamma\in f(\alpha)\Cat=g(\alpha\Cat)$, so that there exists $\beta \in s(\alpha)\Cat$ such that $f(\alpha\beta)=f(\alpha)\gamma$. Thus, $f_{|\alpha} (\beta)=\gamma$, whence $f_{|\alpha}$ is onto, as desired.

Now let $\beta\in s(\alpha)\Cat$ and let $\delta\in s(\beta)\Cat$. Then   $f(\alpha\beta\delta)=f(\alpha\beta)f_{|\alpha\beta}(\delta)=f(\alpha)f_{|\alpha}(\beta\delta)$. On the other hand, $f(\alpha\beta)=f(\alpha)f_{|\alpha}(\beta)$, whence $f(\alpha)f_{|\alpha}(\beta\delta)=f(\alpha)f_{|\alpha}(\beta)f_{|\alpha\beta}(\delta)$. By left cancellation, $f_{|\alpha}(\beta\delta)=f_{|\alpha}(\beta)f_{|\alpha\beta}(\delta)$. Thus, $f_{|\alpha}$  satisfies condition $(2)$. Also, $f_{|\alpha}$ is such that 
$f(\alpha)=f(\alpha s(\alpha))=f(\alpha)f_{|\alpha}(s(\alpha))$, and again by left cancellation $f_{|\alpha}(s(\alpha))=s(f(\alpha))$, so condition $(3)$ is fulfilled. Finally 
$$\dmap(\alpha)\dmap(\beta)=\dmap(\alpha\beta)=\dmap(f(\alpha\beta))=\dmap(f(\alpha)f_{|\alpha}(\beta))=\dmap(f(\alpha))\dmap(f_{|\alpha}(\beta))=\dmap(\alpha)\dmap(f_{|\alpha}(\beta))\,$$
for every $\beta\in s(\alpha)\Cat$. Thus, $\dmap(\beta)=\dmap(f_{|\alpha}(\beta))$ for every  $\beta\in s(\alpha)\Cat$, so condition $(4)$ is satisfied.
\end{proof}
  
\begin{lemma}\label{cocycle_prop}
	Let $\Cat$ be a LCSC with length function $\dmap:\Lambda\to\Gamma$. Then, given $\alpha,\beta\in \Cat$ with $s(\alpha)=r(\beta)$ and $f\in \PisoG$, we have that:
	\begin{enumerate}
		\item $f_{|d(f)}=f$,
		\item $s(f(\alpha))=f_{|\alpha}(s(\alpha))=c(f_{|\alpha})$,
		\item $\text{id}_{s(\alpha)}=(\text{id}_{r(\alpha)})_{|\alpha}$,
		\item $f_{|\alpha\beta}=(f_{|\alpha})_{|\beta}$,
	\end{enumerate} 
\end{lemma}
\begin{proof}
$(1)-(3)$ are straightforward by the definition. To prove $(4)$, let  $\alpha,\beta,\gamma\in \Cat$ with $s(\alpha)=r(\beta)$ and $s(\beta)=r(\gamma)$, and  let $f\in \PisoG$. Then 
$$f(\alpha\beta\gamma)=f(\alpha)f_{|\alpha}(\beta\gamma)=f(\alpha)f_{|\alpha}(\beta)(f_{|\alpha})_{|\beta}(\gamma)\,,$$
and on the other hand
$$f(\alpha\beta\gamma)=f(\alpha\beta)f_{|\alpha\beta}(\gamma)=f(\alpha)f_{|\alpha}(\beta)f_{|\alpha\beta}(\gamma)\,.$$
By left cancellation of $\Cat$, we have that $(f_{|\alpha})_{|\beta}(\gamma)=f_{|\alpha\beta}(\gamma)$. Since this equality holds for every $\alpha\in s(\beta)\Cat$, it follows that $f_{|\alpha\beta}=(f_{|\alpha})_{|\beta}$.
\end{proof}
 
\vspace{.2truecm}

Next step is to show that $\PisoG$ has a natural groupoid structure.
\begin{lemma}\label{Piso_comp}
Let $\Cat$ be a LCSC with length function $\dmap:\Lambda\to\Gamma$. Then, given $f,g\in \PisoG$ such that $d(f)=c(g)$,  the composition $f\circ g\in \PisoG$.	
\end{lemma}
\begin{proof}
	Given $f,g\in \PisoG$ such that $d(f)=c(g)$, we have that $(f\circ g):d(g)\Cat \to c(f)\Cat$ is given by $f(g(\alpha))$ for every $\alpha\in d(g)\Cat$. Clearly, $f\circ g$ is a bijection and $(f\circ g)(d(g))=c(f)$. It remains to check that $f\circ g$ satisfies condition $(2)$. Let $\alpha\in d(g)\Cat$ and $\beta\in s(\alpha)\Cat$. Then $(f\circ g)(\alpha\beta)=f(g(\alpha\beta))=f(g(\alpha)g_{|\alpha}(\beta) )=f(g(\alpha))f_{|g(\alpha)}(g_{|\alpha}(\beta))=(f\circ g)(\alpha)f_{|g(\alpha)}(g_{|\alpha}(\beta))$. Finally, since by Lemma \ref{cocycle} the maps $g_{|\alpha}$ and $f_{|g(\alpha)}$ are surjective it follows that $(f\circ g)(\alpha\Cat)=(f\circ g)(\alpha)\Cat$
\end{proof} 

For each $v\in \Cat^0$, we define $\text{id}_v: v\Cat \rightarrow v\Cat$ to be the identity map on $v\Cat$, so that $\text{id}_v\in \PisoG$

\begin{lemma}\label{Piso_inv}
	Let $\Cat$ be a LCSC  with length function $\dmap:\Lambda\to\Gamma$, and let $f\in \PisoG$. Then  $f^{-1}\in \PisoG$.	
\end{lemma}

\begin{proof}
	If $f\in \text{PIso}(\Cat)$, then $f^{-1}:c(f)\Cat \to d(f)\Cat$ is the unique function such that $(f^{-1}\circ f)=\text{id}_{d(f)}$ and $(f\circ f^{-1})=\text{id}_{c(f)}$. Clearly, $f^{-1}$ satisfies conditions $(1)$ and $(3)$. Let $\alpha\in c(f)\Cat$, and $\beta\in s(\alpha)\Cat$. Then $f^{-1}(\alpha\beta)=f^{-1}(f(f^{-1}(\alpha))\beta)$, but since $f(f^{-1}(\alpha)\Cat)=f(f^{-1}(\alpha))\Cat=\alpha\Cat$ and $f$ is bijective, it follows that there exists a unique $\gamma\in s(f^{-1}(\alpha))\Cat$ such that $\alpha\beta=f(f^{-1}(\alpha))\beta=f(f^{-1}(\alpha)\gamma)$. Therefore, $f^{-1}(\alpha\beta)=f^{-1}(f(f^{-1}(\alpha)\gamma))=f^{-1}(\alpha)\gamma$. Moreover since for every $\gamma\in s(f^{-1}(\alpha))\Cat$ there exists $\beta\in s(\alpha)\Cat$ such that $\alpha\beta=f(f^{-1}(\alpha)\gamma)$, it follows that $f^{-1}(\alpha\Cat)=f^{-1}(\alpha)\Cat$.
 \end{proof}

Now, using Lemmas \ref{Piso_comp} \& \ref{Piso_inv}, we can define a groupoid structure in $\PisoG$.
 
 \begin{definition}\label{proposition:PIso-groupoid}
 	If $\Cat$ is  a LCSC with length function $\dmap:\Lambda\to\Gamma$, then $\PisoG$ is a discrete groupoid, where given $f,g\in \PisoG$ the product $fg$ is defined as the  composition $f\circ g$ whenever $d(f)=c(g)$, and $f^{-1}$ is the set-theoretical inverse of $f$.  Moreover, we can identify unit space of $(\PisoG)^{(0)}=\{\text{id}_v:v\in \Cat^0 \}$ with $\Cat^0$.
 \end{definition}

With this in mind, we can define the notion of action of a groupoid on a LCSC that we need.

\begin{definition}\label{definition_GroupoidAction}
Let $\Lambda$ be a LCSC with length function $\dmap:\Lambda\to\Gamma$, and let $\Grpd$ be a discrete groupoid. An action of $\Grpd$ on $\Cat$ is a groupoid homomorphism $\phi: \Grpd  \rightarrow   \PisoG$. Given $g\in \Grpd$ and $\alpha\in d(\phi(g))\Cat$, we write the action of $g$ on $\alpha$ by $g\cdot \alpha:=\phi(g)(\alpha)$.
\end{definition}
Using the identification $(\PisoG)^{(0)}=\Cat^0$, so that $\phi(\Grpdu)\subseteq\Cat^0$, and that $\phi(s(g))=d(\phi(g))$ and $\phi(r(g))=c(\phi(g))$ for any $g\in \Grpd$, we have that  $$\phi(g): \phi(s(g))\Cat \rightarrow \phi(r(g))\Cat\,.$$ \vspace{.2truecm}

\begin{remas}\label{Rem:2.9} Let $\Cat$ be a LCSC  with length function $\dmap:\Lambda\to\Gamma$, and let $\phi:\Grpd\to \PisoG$ be a groupoid action on $\Cat$. Then:
\begin{enumerate} 
	\item  Suppose that $\Delta^0:=\phi(\Grpdu)\subsetneqq \Cat^0$. Let us define  
	$$\Delta:=\{\alpha\in \Cat; s(\alpha),r(\alpha)\in \Delta^0 \}\,.$$ 
	Then $\Delta$ is a LCSC, and $\phi(\Grpd)\subseteq \PisoG$.
	However $\Delta$ is not necessarily finitely aligned if so is $\Cat$.  Thus, such a restriction will affect the arguments. Hence, we will always assume that, given an action of $\Grpd$ on $\Cat$, either $\phi(\Grpdu)=\Cat^0$, or that the restricted subcategory $\Delta$ inherits the finitely aligned property. In practice, that means that, after restricting if necessary, we will assume that $\phi(\Grpdu)= \Cat^0$.
	
	\item  Suppose that $\phi_{|\Grpdu}$ is not an injective map. Then it can happen that there exists $(g,h)\notin \Grpdc$, that is $s(g) \neq r(h)$, but $(\phi(g),\phi(h))\in (\PisoG)^{(2)}$, that is $\phi(s(g))=d(\phi(g))=c(\phi(h))=\phi(r(h))$, and hence $\phi(g)\phi(h)\in \PisoG\setminus \phi(\Grpd)$. This will affect defining composability of elements in the corresponding Zappa-Sz\'ep product. Since our model should include non-faithful self-similar actions of a groupoid $\Grpd$ on $\Cat$, we cannot skip that case. Then, for a non-injective action, we will need to include the condition that for every $g,h\in \Grpd$, then $s(g)=r(h)$ if and only if $d(\phi(g))=\phi(s(g))=\phi(r(h))=c(\phi(h))$, or equivalently, that $\phi_{|\Grpdu}$ is injective. 
	\end{enumerate}
\end{remas}

According to Remarks \ref{Rem:2.9}, we will assume  during the whole paper that an action of a discrete groupoid $\Grpd$ on a small category $\Cat$ through a (not necessarily injective) groupoid homomorphism $\phi: \Grpd\rightarrow \PisoG$, satisfies that $\phi_{|\Grpdu}$ is a bijection. Therefore we will identify $\Grpdu=\Cat^{(0)}=\PisoG^{(0)}$ omitting  $\phi$.

\section{Zappa-Sz\'ep products for groupoid actions}

In this section, we will define the Zappa-Sz\'ep product of a left cancellative small category categories.
This is inspired in the construction of the Zappa-Sz\'ep product of a groupoid on a finite graph \cite[Section 3]{LRRW} and the Zappa-Sz\'ep product of a group acting on a left cancellative small category \cite{BCFS,OP_LCSC}. This has also been recently done in \cite{LV2} where they construct Zappa-Sz\'ep products of groupoids acting on higher-rank graphs. 

First, we will state the abstract notion of self-similar action, in order to boil up the exact definition of 1-cocycle we will need.

\begin{definition}\label{Def:SelfSimilarAction}
Let $\Grpd$ be a discrete groupoid acting on a LCSC $\Cat$ with length function $\dmap:\Lambda\to\Gamma$. We say that the action is \emph{self-similar} if for every $g\in \Grpd$ and for every $\alpha\in s(g)\Cat$ there exists $h\in \Grpd$ such that
$$g\cdot (\alpha\mu)=(g\cdot \alpha)(h\cdot\mu)$$
for every $\mu\in s(\alpha)\Cat$.
\end{definition}

Clearly, in the above definition $h$ depends on $g$ and $\alpha$. A natural question is to decide whether such an $h$ is unique. Suppose that $h_1, h_2\in \Grpd$ satisfy $g\cdot (\alpha\mu)=(g\cdot \alpha)(h_i\cdot \mu)$  ($i=1,2)$ for every $\mu\in s(\alpha)\Cat$. Then
$$(g\cdot \alpha)(h_1\cdot \mu) =(g\cdot \alpha)(h_2\cdot \mu) .$$
Since  $\Lambda$ is left cancellative, then 
$$(h_1\cdot \mu)=(h_2\cdot \mu)$$
for every $\mu\in s(\alpha)\Lambda$. 
Hence,  we can only guarantee that $\phi(h_1)=\phi(h_2)$, and $h_1=h_2$ only holds when $\phi$ is injective.
\vspace{.2truecm}
\ 
We will use this fact to define a suitable notion of cocycle for such an action. We will essentially follow \cite[Section 4]{BKQS}, taking care of the fact that the action is partial.

\begin{definition}\label{Def:GroupoidCocycle}
Let $\Cat$ be a LCSC with length function $\dmap:\Lambda\to\Gamma$, and let $\Grpd$ be a discrete groupoid acting on $\Cat$. Consider the set
$$\Grpd{}_s\times_r\Cat:=\{ (g,\alpha)\in \Grpd\times \Cat : s(g)=r(\alpha)\}\,.$$
A \emph{(partial) cocycle of the action of $\Grpd$ on $\Cat$} is a function $\varphi :  \Grpd{}_s\times_r\Cat \rightarrow  \Grpd$ satisfying the cocycle identity $$\varphi (gh,\alpha)=\varphi(g, h\cdot\alpha)\varphi(h,\alpha)$$ for any $(g,h)\in \Grpdc$ 
 and any $\alpha \in \Cat$ such that $s(h)=r(\alpha)$.
 
 Observe that in particular $\varphi(r(\alpha),\alpha)\in \Grpdu$ for every $\alpha\in \Cat$.
\end{definition}

Now, we state the properties that a cocycle will enjoy (in a similar list as that of \cite[Section 7]{OP_LCSC}). But instead of impose them, we will try to deduce from the definition, as in \cite[Section 3]{LRRW}.

First step is to fix the requirement to guarantee that the action of $\Grpd$ is compatible with the composition in $\Cat$. 
We will introduce minimal requirements to guarantee this fact when constructing our ``self-similar'' actions, and also that we can associate a suitable small category to the action and the cocycle.

\begin{definition}\label{Def:CategoryCocycle}
Let $\Cat$ be a LCSC with length function $\dmap:\Lambda\to\Gamma$, and let $\Grpd$ be a discrete groupoid acting on $\Cat$. A (partial) cocycle $\varphi$ for the action of $\Grpd$ on $\Cat$ is said to be a \emph{category cocycle} if for every $g\in \Grpd$, $\alpha\in\Cat$ with $s(g)=r(\alpha)$ and every $\beta\in s(\alpha)\Cat$ we have that:
\begin{enumerate}
\item $\varphi(g,d(\phi(g)))=g$,	
\item $s(g\cdot \alpha)=\varphi(g,\alpha)\cdot s(\alpha)=r(\varphi(g,\alpha))$,
\item $s(\alpha)=\varphi(r(\alpha),\alpha)$,
\item $\varphi(g,\alpha\beta)=\varphi(\varphi(g,\alpha), \beta)$,
\item $g\cdot (\alpha\beta)=(g\cdot \alpha)(\varphi(g,\alpha)\cdot \beta)$.
\end{enumerate}
If we have an action of $\Grpd$ on $\Cat$, and $\varphi$ is a category cocycle for this action, we will say that $(\Cat,\dmap, \Grpd, \varphi)$ is a \emph{category system}.
\end{definition}

\begin{remark}
	Let $\Cat$ be a LCSC with length function $\dmap:\Lambda\to\Gamma$, let $\Grpd$ be a discrete groupoid acting on $\Cat$, and let $(\Cat,\dmap, \Grpd, \varphi)$ be a category system. Then, by condition $(5)$ in Definition \ref{Def:CategoryCocycle},
the action of $\Grpd$ on $\Cat$ is self-similar. Moreover, 	
	given $g\in \Grpd$ and $\alpha\in s(g)\Cat$, we have that $\phi(\varphi(g,\alpha))=\phi(g)_{|\alpha}$. Indeed, if $\beta\in s(\alpha)\Cat$ then $\phi(g)(\alpha\beta)=\phi(g)(\alpha)\phi(g)_{|\alpha}(\beta)$. On the other hand  $\phi(g)(\alpha\beta)=g\cdot(\alpha\beta)=(g\cdot \alpha)(\varphi(g,\alpha)\cdot \beta)=\phi(g)(\alpha)(\phi(\varphi(g,\alpha))(\beta))$. By left cancellation of $\Cat$ it follows that $\phi_{|\alpha}(\beta)=\phi(\varphi(g,\alpha))(\beta)$, and hence $$\phi(g)_{|\alpha}=\phi(\varphi(g,\alpha))\,.$$
	
\end{remark}
\begin{example}
	If $\Cat$ is a LCSC with length function $\dmap:\Lambda\to\Gamma$, then $\varphi:\PisoG {}_d\times_r \Cat\to \PisoG$ defined by $\varphi(f,\alpha)=f_{|\alpha}$ is a category cocyle (Lemmas \ref{cocycle} \& \ref{cocycle_prop}), so $(\Cat,\dmap,\PisoG,\varphi)$ is a category system. 
\end{example}
The property of being a category cocycle is the correct version of \cite[(2.3)]{EP2}, as we will see. Let us determine which are the basic properties satisfied by a category cocycle. These properties, joint with Definition \ref{Def:CategoryCocycle}, are analog to those proved in \cite[Lemma 3.4 \& Proposition 3.6]{LRRW}.

\begin{proposition}\label{Prop:left cancellative small category}
Let $\Cat$ be a LCSC with length function $\dmap:\Lambda\to\Gamma$, let $\Grpd$ be a  discrete groupoid acting on $\Cat$, and let $(\Cat,\dmap, \Grpd, \varphi)$ a category system. Then, for every $g\in \Grpd$, every $\alpha, \beta\in \Cat$ with $s(g)=r(\alpha)$, $s(\alpha)=r(\beta)$, we have that:
\begin{enumerate}
\item $r(g\cdot \alpha)=g\cdot r(\alpha)$.
\item $s(\varphi(g,\alpha))=s(\alpha)$.
\item $\varphi(g,\alpha)^{-1}=\varphi(g^{-1}, g\cdot \alpha)$.
\end{enumerate}
\end{proposition}
\begin{proof}
Let $g\in \Grpd$, let $\alpha, \beta\in \Cat$ with $s(g)=r(\alpha)$, $s(\alpha)=r(\beta)$. 

(1) Since $r(\alpha)=s(g)$ and $g\cdot s(g)=r(g)$, we have that $r(g\cdot \alpha)=r(g)=g\cdot s(g)=g\cdot r(\alpha)$.

(2) By Definition \ref{Def:CategoryCocycle}(2), $\varphi(g,\alpha)\cdot s(\alpha)=r(\varphi(g,\alpha))$. Thus, $s(\varphi(g,\alpha))=s(\alpha)$.

(3) By the cocycle condition and Definition \ref{Def:CategoryCocycle}(3)   $$s(\alpha)=\varphi(g^{-1}g,\alpha)=\varphi(g^{-1},g\cdot \alpha)\varphi(g,\alpha)\,,$$ whence $\varphi(g,\alpha)^{-1}=\varphi(g^{-1}, g\cdot \alpha)$.

\end{proof}

Now, we will fix a definition of Zappa-Sz\'ep product for a category system $(\Cat,\Grpd,\varphi)$, similar to that in \cite[Section 4]{BKQS}, using the strategy introduced in \cite{BPRRW}. 

\begin{definition}\label{definition_zappa-szed2}
Let $\Cat$ be a LCSC with length function $\dmap:\Lambda\to\Gamma$, let $\Grpd$ be a  discrete groupoid acting on $\Cat$, and let $(\Cat,\dmap,\Grpd,\varphi)$ be a category system. We denote by $\Cat \Join^\varphi \Grpd$ the set
$$\Cat \Join^\varphi \Grpd:=\Cat{}_s\times_r\Grpd=\{ (\alpha, g)\in \Cat \times \Grpd : s(\alpha)=r(g)\}\,,$$
with distinguished elements
$$(\Cat \Join^\varphi \Grpd)^{(0)}:= \Cat^0 \Join^\varphi \Grpdu\,.$$
We equip this pair of sets with range and source maps $r,s:\Lambda \Join^\varphi G \rightarrow (\Lambda \Join^\varphi G)^{(0)}$ defined by 
$$r(\alpha, g):=(r(\alpha), r(\alpha))\qquad \text{and}\qquad s(\alpha, g):=(s(g),s(g))=(g^{-1}\cdot s(\alpha), s(g))\,,$$
for every $(\alpha, g)\in \Cat \Join^\varphi \Grpd$. Observe that the later equality holds because $g\cdot s(g)=r(g)=s(\alpha)$, whence $g^{-1}\cdot s(\alpha)=g^{-1}\cdot r(g)=s(g)$.

Moreover, given $(\alpha, g), (\beta, h)\in \Cat \Join^\varphi \Grpd$ with $s(\alpha, g)=r(\beta, h)$, we define the composition of two elements as follows:
$$(\alpha, g)(\beta, h):=(\alpha(g\cdot \beta), \varphi(g,\beta)h).$$
\end{definition}

We now adapt the arguments of \cite[Proposition 4.6]{BKQS} and \cite[Proposition 4.13]{BKQS} to prove the next results. 

\begin{proposition}
	Let $\Cat$ be a LCSC with length function $\dmap:\Lambda\to\Gamma$, let $\Grpd$ be a discrete groupoid acting on $\Cat$, and let $(\Cat,\dmap,\Grpd,\varphi)$ be a category system. If   given $(\alpha, g), (\beta, h)\in \Cat \Join^\varphi \Grpd$ with $s(\alpha, g)=r(\beta, h)$, we define the composition of these  two elements by
	$$(\alpha, g)(\beta, h):=(\alpha(g\cdot \beta), \varphi(g,\beta)h)\,,$$
	then $\Cat \Join^\varphi \Grpd$ is a small category.
\end{proposition}
\begin{proof}
	First observe that, given $(\alpha, g), (\beta, h)\in \Cat \Join^\varphi \Grpd$ with $s(\alpha, g)=r(\beta, h)$,
	$$r(\beta)=g^{-1}\cdot s(\alpha)\qquad \text{and}\qquad r(\beta)=s(g)\,.$$
	Whence $g\cdot \beta$ is well-defined, and since $r(\beta)=g^{-1}\cdot s(\alpha)$ so does $\alpha(g\cdot \beta)$.

	On the other hand we have that $$s(\alpha(g\cdot \beta))=s(g\cdot \beta)=r(\varphi(g,\beta))=r(\varphi(g,\beta)h)$$ and 
	$$s(\varphi(g,\beta))=s(\varphi(g,\beta))=s(\beta)=r(h)\,.$$ 
	Then, the product 
	$$(\alpha, g)(\beta, h):=(\alpha(g\cdot \beta), \varphi(g,\beta)h)\,.$$
	is well-defined.

	Now, observe that 
	\begin{align*}
	r((\alpha,g)(\beta,h)) & = r((\alpha(g\cdot \beta), \varphi(g,\beta)h))=(r(\alpha(g\cdot \beta)),r(\alpha(g\cdot \beta))) \\& = (r(\alpha),r(\alpha ))= r(\alpha,g)\,,
	\end{align*}
	and 
		\begin{align*}
	s((\alpha,g)(\beta,h)) & = s((\alpha(g\cdot \beta), \varphi(g,\beta)h))=(s(\varphi(g,\beta)h), s(\varphi(g,\beta)h)) \\& = (s(h), s(h))= s(\beta,h)\,.
	\end{align*}
	Also, given $(\gamma,f)\in \Cat \Join^\varphi \Grpd$, we have that 
	\begin{align*}
	((\alpha,g)(\beta,h))(\gamma,f) & = (\alpha(g\cdot \beta), \varphi(g,\beta)h)(\gamma,f) \\ &  = (\alpha(g\cdot \beta)(\varphi(g,\beta)h\cdot \gamma), \varphi(\varphi(g,\beta)h,\gamma)f) \\
	 & =(\alpha(g\cdot (\beta(h\cdot \gamma)), \varphi(\varphi(g,\beta)h,\gamma)f) \\
	 	 & =(\alpha(g\cdot (\beta(h\cdot \gamma)), \varphi(\varphi(g,\beta),h\cdot \gamma)\varphi(h, \gamma)f) \\
	 	 	 	 & = ( \alpha(g\cdot (\beta(h\cdot \gamma)),\varphi(g,\beta (h\cdot \gamma))\varphi(h, \gamma)f ) \\
	 	 & = (\alpha,g)(\beta(h\cdot \gamma),\varphi(h, \gamma)f ) \\
	 	  & = (\alpha,g)((\beta, h)(\gamma,f))\,.
	\end{align*}
	
Now, let $(u,u)\in \Cat^0 \Join^\varphi\Grpdu$, so $u\in \Cat^0= \Grpdu$. Then, 
	$$r(u,u)=(u,u)\qquad\text{and}\qquad s(u,u)=(u,u)\,.$$
	Finally, 
	\begin{align*}
	r(\alpha,g)(\alpha,g) = & (r(\alpha),r(\alpha))(\alpha,g)= (r(\alpha)(r(\alpha)\cdot \alpha), \varphi(r(\alpha),\alpha)g) =(\alpha,g)\,,
	\end{align*}
	and
	\begin{align*}
	(\alpha,g)s(\alpha,g) =& (\alpha,g)(s(g),s(g))=(\alpha(g\cdot s(g)), \varphi(g,s(g))s(g))  \\ 
	=&  (\alpha r(g), gs(g)) =(\alpha,g)\,,
	\end{align*}
	So we are done
\end{proof}

\begin{lemma}\label{lemma_inv}
	Let $\Cat$ be a LCSC with length function $\dmap:\Lambda\to\Gamma$, let $\Grpd$ be a discrete groupoid acting on $\Cat$, and let $(\Cat, \dmap,\Grpd,\varphi)$ be a category system. Then, $(\Cat \Join^\varphi \Grpd)^{-1}=\Cat^{-1}{}_s\times_r\Grpd$.
\end{lemma}
\begin{proof}
	If $(\alpha,g)\in (\Cat \Join^\varphi \Grpd)^{-1}$, then there exists $(\beta,h)\in \Cat \Join^\varphi \Grpd$ such that 
	$$(\alpha,g)(\beta,h)=(\alpha(g\cdot \beta),\varphi(g,\alpha)h)=r(\alpha,g)=(r(\alpha),r(\alpha))\,.$$
	Therefore, $\alpha(g\cdot \beta)=r(\alpha)$, and so $\alpha\in \Cat^{-1}$. Hence, $(\alpha,g)\in \Cat^{-1}{}_s\times_r\Grpd$. If $(\alpha,g)\in \Cat^{-1}{}_s\times_r\Grpd$,  then 
	$$(\alpha,g)(g^{-1}\cdot \alpha^{-1},\varphi(g,g^{-1}\cdot\alpha^{-1})^{-1})=(r(\alpha),r(\alpha))\,,$$
	and so $(\alpha,g)\in (\Cat \Join^\varphi \Grpd)^{-1}$.
\end{proof}

\begin{proposition}\label{proposition_LC and FA}
Let $\Cat$ be a LCSC with length function $\dmap:\Lambda\to\Gamma$, let $\Grpd$ be a discrete groupoid acting on $\Cat$, and let $(\Cat,\dmap,\Grpd,\varphi)$ be a category system. Then:
\begin{enumerate}
\item $\Cat\Join^\varphi \Grpd$ is left cancellative.
\item If $\Cat$ is finitely (singly) aligned, then:
\begin{enumerate}
\item $(\alpha, g)(\Cat\Join^\varphi \Grpd) \cap (\beta,h)(\Cat\Join^\varphi \Grpd)=(\alpha\Cat \cap \beta\Cat)_s\times _r \Grpd$.
\item $\Cat\Join^\varphi \Grpd$ is finitely (singly) aligned.
\end{enumerate}
\end{enumerate}
\end{proposition}
\begin{proof} $\mbox{ }$

(1) By hypothesis, $\Cat$ is left cancellative. Let $(\alpha, g), (\beta ,h), (\gamma, k)\in \Cat \Join^\varphi \Grpd$ such that $(\alpha, g)(\beta, h)=(\alpha, g)(\gamma, k)$. Then, $(\alpha(g\cdot \beta), \varphi(g, \beta)h)=(\alpha(g\cdot \gamma), \varphi(g, \gamma)k)$. Thus, we have that $\alpha(g\cdot \beta)= \alpha(g\cdot \gamma)$, and by left cancellation we get $g\cdot \beta= g\cdot \gamma$, whence $\beta=\gamma$. Moreover, $\varphi(g, \beta)h = \varphi(g, \gamma)k = \varphi(g, \beta)k$, and so $h=k$. Thus, $(\beta ,h) = (\gamma, k)$, as desired.

(2) Let us prove both asserts:
\begin{enumerate}
\item[(a)] First, let $(\alpha, g),(\beta,h)\in \Cat \Join^\varphi \Grpd$ such that $(\alpha,g)\Cap (\beta,h)$. This means that there exist $(\gamma,k),(\delta,l)\in  \Cat \Join^\varphi \Grpd$ such that $z:=(\alpha, g)(\gamma, k)=(\beta, h)(\delta, l)$, whence $$z=(\alpha(g\cdot \gamma), \varphi(g, \gamma)k)=( \beta(h\cdot \delta), \varphi(h, \delta)l)\,.$$ In particular, $\alpha(g\cdot \gamma)=\beta(h\cdot \delta)\in \alpha\Lambda \cap \beta\Lambda$. Moreover, since $z\in \Cat \Join^\varphi \Grpd$, we conclude that $z\in (\alpha\Cat \cap \beta\Cat)_s\times _r \Grpd$.

Conversely, suppose that $z\in (\alpha\Cat \cap \beta\Cat)_s\times _r \Grpd$. Then, $z=(\varepsilon, m)$ with $\varepsilon \in \alpha\Cat \cap \beta\Cat$, $m\in \Grpd$ and $s(\varepsilon)=c(\phi(m))$. Set $\varepsilon=\alpha\lambda$, and notice that, for any $(\alpha, g)\in \Cat \Join^\varphi \Grpd$, $(\alpha,g)=(\alpha, s(\alpha))(s(\alpha), g)\in \Cat \Join^\varphi \Grpd$, with $(s(\alpha), g)\in (\Cat \Join^\varphi \Grpd)^{-1}$ with inverse $(g^{-1}\cdot s(\alpha), g^{-1})$. Thus, $(\alpha, \text{id}_{s(\alpha)})\approx (\alpha, g)$ in $\Cat$. Hence, $z=(\alpha\lambda, m)=(\alpha, s(\alpha))(\lambda, m)\in (\alpha, s(\alpha))(\Cat \Join^\varphi \Grpd)=(\alpha, g)(\Cat \Join^\varphi \Grpd)$. Similarly, we prove that $z\in (\beta, h)(\Cat \Join^\varphi \Grpd) $. Thus, $z\in (\alpha, g)(\Cat \Join^\varphi \Grpd) \cap (\beta,h)(\Cat \Join^\varphi \Grpd)$, so we are done.

\item[(b)] By part (a), for any $\gamma\in \Cat$ we have $\gamma\Cat{}_s\times _r \Grpd=(\gamma, s(\gamma))(\Cat \Join^\varphi \Grpd)$. 

Suppose that $\Lambda$ is finitely (singly) aligned. Then, using part (a), we have that
\begin{align*}
(\alpha, g)(\Cat \Join^\varphi \Grpd) \cap (\beta,h)(\Cat \Join^\varphi \Grpd) &   = (\alpha\Cat \cap \beta\Cat){}_s\times _r \Grpd =   \left(\bigcup\limits_{\gamma\in \alpha\vee \beta} \gamma\Cat \right) { }_s\times_c \Grpd \\
& =\bigcup\limits_{\gamma\in \alpha\vee \beta}\left( \gamma\Cat {}_s\times_r G\right)=\bigcup_{\gamma\in \alpha\vee \beta}(\gamma, s(\gamma))(\Cat\Join^\varphi \Grpd) \\
& = \bigcup_{\eta\in (\alpha\vee\beta){}_s\times_r \Cat^0}\eta(\Cat\Join^\varphi \Grpd)\,.
\end{align*}
Thus, $\Cat\Join^\varphi \Grpd$ is finitely (singly) aligned, as desired.
\end{enumerate}
\end{proof}

Also, we can prove an analog of \cite[Lemma 4.15]{BKQS}

\begin{lemma}\label{lemma_FiniteCovers}
Let $\Cat$ be a LCSC with length function $\dmap:\Lambda\to\Gamma$, let $\Grpd$ be a discrete groupoid acting on $\Cat$, and let $(\Cat,\dmap,\Grpd,\varphi)$ be a category system. Take any $(v, v)\in (\Cat\Join^\varphi \Grpd)^0$, and let $F\subseteq (v, v)(\Cat\Join^\varphi \Grpd)$. Set
$$H:=\{ \alpha \in v\Cat : \text{ there is } g\in \Grpd \text{ such that } (\alpha, g)\in F\}.$$
Then, $F$ is exhaustive at $(v, v)$ if and only if $H$ is exhaustive at $v$.
\end{lemma}
\begin{proof}
Since $(\alpha, s(\alpha)) \approx (\alpha ,g)$ for all $(\alpha, g)\in \Cat\Join^\varphi \Grpd$, by \cite[Lemma 2.3]{BKQS} we can assume without loss of generality that $F=H{}_s\times_r \Cat^0$. 

First, suppose that $F$ is exhaustive at $(v, v)$. If $\beta\in v\Cat$, then $(\beta, s(\beta))\in (v, v)(\Cat\Join^\varphi \Grpd)$. Then, there exists $(\alpha, s(\alpha))\in F$ such that 
$$\emptyset\ne (\alpha, s(\alpha))(\Cat\Join^\varphi \Grpd) \cap (\beta, s(\beta))(\Cat\Join^\varphi \Grpd)=(\alpha\Cat \cap \beta\Cat){}_s\times _r \Grpd\,.$$
Hence, $\alpha\Cat \cap \beta\Cat\ne \emptyset$, and since $\alpha\in H$, we conclude that
$H$ is exhaustive at $v$.

Conversely, suppose that $H$ is exhaustive at $v$. Set $(\beta, g)\in (v, v)(\Cat\Join^\varphi \Grpd)$. Then, $\beta\in v\Cat$, so that there exists $\alpha \in H$ such that $\alpha\Cat \cap \beta\Cat\ne \emptyset$. Since $(\beta, g)\approx (\beta, s(\beta))$ in $\Cat\Join^\varphi \Grpd$, we have that
$$
(\alpha, s(\alpha))(\Cat\Join^\varphi \Grpd) \cap (\beta, g)(\Cat\Join^\varphi \Grpd)$$
$$ = (\alpha, s(\alpha))(\Cat\Join^\varphi \Grpd) \cap (\beta, s(\beta))(\Cat\Join^\varphi \Grpd)$$
$$ = (\alpha\Cat \cap \beta\Cat){}_s\times _r \Grpd \ne \emptyset \,.
$$
 
Since $(\alpha, s(\alpha))\in F$, we conclude that $F$ is exhaustive at $(v, v)$, as desired.
\end{proof}

\begin{proposition}\label{ZP_WFP}
Let $\Cat$ be a LCSC with length function $\dmap:\Lambda\to\Gamma$ satisfying the WFP, let $\Grpd$ be a discrete groupoid acting on $\Cat$, and let $(\Cat,\dmap,\Grpd,\varphi)$ be a category system.  Then the map $\dmap: \Cat\Join^\varphi \Grpd\to \Gamma$ defined by $\dmap(\alpha,g)=\dmap(\alpha)$ for every $(\alpha,g)\in \Cat\Join^\varphi \Grpd$ is a length function satisfying the WFP.
\end{proposition}
\begin{proof}
The map $\dmap: \Cat\Join^\varphi \Grpd\to \Gamma$  given by $\dmap(\alpha,g)=\dmap(\alpha)$ for every $(\alpha,g)\in \Cat\Join^\varphi \Grpd$ is clearly well-defined. Now let  $(\alpha, g),(\beta,h)\in \Cat \Join^\varphi \Grpd$ with $s(\alpha,g)=r(\beta,h)$. Then, 
\begin{align*}
\dmap((\alpha, g)(\beta,h)) & = \dmap(\alpha(g\cdot \beta),\varphi(g,\beta)h)=\dmap(\alpha(g\cdot \beta)) \\
& =\dmap(\alpha)\dmap(g\cdot \beta)=\dmap(\alpha)\dmap(\phi(g)(\beta)) \\
& =\dmap(\alpha)\dmap(\beta)=\dmap(\alpha,g)\dmap(\beta,h)\,,
\end{align*}
since $\phi(g)\in \PisoG$. So, $\dmap$ is a length function of $\Cat\Join^\varphi \Grpd$.

Now, let $(\alpha,g)\in \Cat\Join^\varphi \Grpd$, with $\dmap(\alpha,g)=\dmap(\alpha):=\gamma$. Then, by the WFP,  given $\gamma_1,\gamma_2\in\Gamma$ such that $\gamma=\gamma_1\gamma_2$, there exist $\alpha_1,\alpha_2\in \Cat$ with $\alpha=\alpha_1\alpha_2$ and $\dmap(\alpha_1)=\gamma_1$ and $\dmap(\alpha_2)=\gamma_2$. Thus, we have that 
$$(\alpha,g)=(\alpha_1,s(\alpha_1))(\alpha_2,g)\,,$$
and 
$$\gamma=\dmap(\alpha)=\dmap(\alpha,g)=\dmap(\alpha_1,s(\alpha_1))\dmap(\alpha_2,g)=\dmap(\alpha_1)\dmap(\alpha_2)=\gamma_1\gamma_2\,.$$
Let us suppose there exist $(\beta_1,h_1),(\beta_2,h_2)\in \Cat\Join^\varphi \Grpd$  with  
$$(\beta_1,h_1)(\beta_2,h_2)=(\beta_1(h_1\cdot \beta_2),\varphi(h_1,\beta_2)h_2)=(\alpha,g)$$
such that $\dmap(\beta_1,h_1)=\dmap(\beta_1)=\gamma_1$ and $\dmap(\beta_2,h_2)=\dmap(\beta_2)=\gamma_2$. By the WFP there exists $f_1,f_2\in \Cat^{-1}$ such that $\beta_1=\alpha_1f_1$ and $h_1\cdot \beta_2=f_2\alpha_2$, whence
$$(\beta_1,h_1)=(\alpha_1f_1,h_1)=(\alpha_1,s(\alpha_1))(f_1,h_1)$$
and 
\begin{align*}
(\beta_2,h_2)& = (h_1^{-1}\cdot(h_1\cdot \beta_2),\varphi(h_1,\beta_2)^{-1}g  ) \\
& =(h_1^{-1}\cdot(h_1\cdot \beta_2),\varphi(h_1,h_1^{-1}\cdot f_2\alpha_2)^{-1}g  )\\
& = ((h_1^{-1}\cdot (f_2\alpha_2),\varphi( h_1^{-1},f_2 \alpha_2) g) \\
& = ((h_1^{-1}\cdot f_2)(\varphi(h_1^{-1},f_2)\cdot \alpha_2),\varphi(\varphi(h_1^{-1},f_2),\alpha_2) g) \\
& = (h_1^{-1}\cdot f_2, \varphi(h_1^{-1},f_2))(\alpha_2,g) \\
& = (h_1^{-1}\cdot f_2, \varphi(h_1,h_1^{-1}\cdot f_2)^{-1})(\alpha_2,g) \\
& = (f^{-1}_2, h_1)^{-1}(\alpha_2,g)\,,
\end{align*}
as desired.
\end{proof}

\begin{remark}
	Observe that in the proof of the above Proposition that if $\Cat$ satisfies the WFP as it is defined in \cite{LV2} then so does the Zappa-Sz\'ep product. 
\end{remark}

\begin{definition}
Let $\Cat$ be a LCSC with length function $\dmap:\Lambda\to\Gamma$, and let $\Grpd$ be a discrete groupoid acting on $\Cat$. We say that a category system  $(\Cat,\dmap,\Grpd,\varphi)$ is \emph{pseudo-free} if $\Cat$ is an action-free category, and  whenever $g\cdot \alpha=f\alpha$ and $\varphi(g,\alpha)=s(\alpha)$ for some $\alpha\in \Cat$, $g\in \Grpd$ and $f\in \Cat^{-1}$, then $g\in \Grpdu$.
\end{definition}

\begin{proposition}\label{ZP_RC}
Let $\Cat$ be a LCSC with length function $\dmap:\Lambda\to\Gamma$ satisfying the WFP, let $\Grpd$ be a discrete groupoid acting on $\Cat$, and let $(\Cat,\dmap,\Grpd,\varphi)$ be a category system.  Then, $\Cat\Join^\varphi \Grpd$ is right cancellative if and only if  $(\Cat,\dmap,\Grpd,\varphi)$ is pseudo-free.
\end{proposition}
\begin{proof}
First suppose that $\Cat\Join^\varphi \Grpd$ is right cancellative. Then $\Cat$ is right cancellative too, and hence action-free. Now, let us suppose there exist $\alpha\in \Cat$, $g\in \Grpd$ and $f\in \Cat^{-1}$ such that  $g\cdot \alpha=f\alpha$ and $\varphi(g,\alpha)\in \Grpdu$. Thus, 
$$(r(g),g)(\alpha,s(\alpha))=(f\alpha,s(\alpha))=(f,s(f))(\alpha,s(\alpha))\,,$$
but by right cancellation $g=s(f)\in \Grpdu$. Hence, $(\Cat,\dmap,\Grpd,\varphi)$ is pseudo-free.

On the other hand, suppose that $(\Cat,\dmap,\Grpd,\varphi)$ is pseudo-free, and let $(\alpha,g)\in \Cat\Join^\varphi \Grpd$ and $(f,h)\in (\Cat\Join^\varphi \Grpd)^{-1}$  such that
$$(f,h)(\alpha,g)=(f(h\cdot \alpha),\varphi(h,\alpha)g)=(\alpha,g)\,,$$
so $\alpha=f(h\cdot \alpha)$ and $\varphi(h,\alpha)=r(g)$, and hence  $f^{-1}\alpha=h\cdot \alpha$. But then, by pseudo-freeness, we have that $h=r(\alpha)$. Hence, $f^{-1}\alpha=\alpha$, but since $\Cat$ is action-free we have $f=r(\alpha)$. Therefore, $(f,h)=(r(\alpha),r(\alpha))\in (\Cat\Join^\varphi \Grpd)^0$, so $\Cat\Join^\varphi \Grpd$ is action-free, and hence right cancellative. 
\end{proof}

\section{Length functions in LCSC}

Now we will see that Zappa-Sz\'ep products of LCSC arises naturally among LCSC. In \cite{LV} they defined the \emph{generalized higher rank $k$-graphs} categories and described them as Zappa-Sz\'ep products. Here we slightly generalize their arguments to a more general class of LCSC. 

Let $\Cat$ be a LCSC, and let $M$ be a monoid with neutral element $1$, and  let $\dmap:\Cat\to M$ be a length function satisfying:
\begin{enumerate}
	\item[(LF1)] $\dmap^{-1}(1)=\Cat^{-1}$,
	\item[(LF2)] for every $\alpha\in \Cat$, if $\dmap(\alpha)=m_1m_2$ for $m_1,m_2\in M$ then there exist $\alpha_1,\alpha_2\in M$ such that $\alpha=\alpha_1\alpha_2$ and $\dmap(\alpha_i)=m_i$ for $i=1,2$.
\end{enumerate}

Observe that if $\dmap:\Cat\to M$ satisfies the WFP then it satisfies conditions (LF1) and (LF2).  

\begin{remark} A length function $\dmap:\Cat\to M$ satisfying the conditions (LF1) and (LF2) forces the monoid $M$ to be conical. 
	Let $m_1,m_2\in M$ such that $1=m_1 m_2$. Take any $\alpha\in \Cat^{-1}$, so $1=\dmap(\alpha)=m_1 m_2$, and hence by (LF2) there exist $\alpha_1,\alpha_2$ such that $\alpha=\alpha_1\alpha_2$ with $\dmap(\alpha_i)=m_i$ for $i=1,2$. Thus, $s(\alpha)=(\alpha^{-1}\alpha_1)\alpha_2$, so that $\alpha_2\in \Lambda^{-1}$. Hence, $\alpha_1=\alpha\alpha_2^{-1}\in \Lambda^{-1}$. By (LF1), $m_i=\dmap(\alpha_i)=0$ for $i=1,2$. Thus, $M$ is conical. 
	
\end{remark}

We denote by $e\in M$ is an \emph{atom} if whenever $e=m_1m_2$ then either $m_1=1$ or $m_2=1$. We denote by $M_a$ the set of all atomic elements of $M$. We say that $M$ is \emph{atomic} if $M$ is generated by its atoms.
We say that $\alpha\in \Cat$ is an \emph{atom} if whenever $\alpha=\beta\gamma$ then either $\beta$ or $\gamma$ is in $\Cat^{-1}$. We denote by $\Cat_a$ the atoms of $\Cat$.

\begin{lemma}\label{lemma_1}
	Let $\Cat$ be a LCSC, let $M$ be a conical monoid, and let $\dmap:\Cat\to M$ be a length function satisfying (LF1) and (LF2). Then, $\alpha$ is an atom if and only if $\dmap(\alpha)$ is an atom. In particular, $\Cat_a=\bigcup_{e\in M_a}\dmap^{-1}(e)$. 
\end{lemma}
\begin{proof}
	Let $e\in M$ be an atom, and let $\alpha\in \Cat$ with $\dmap(\alpha)=e$. If $\alpha=\beta\gamma$, then $e=\dmap(\alpha)=\dmap(\beta\gamma)=\dmap(\beta)\dmap(\gamma)$. Hence, either $\dmap(\beta)=1$ or $\dmap(\gamma)=1$, or equivalently either $\beta\in \Cat^{-1}$ and $\gamma\in \Cat^{-1}$. Thus, $\alpha$ is an atom.
	
	Now, let $\alpha\in \Cat$ be an atom of $\Cat$. Suppose that $\dmap(\alpha)=m$ is not an atom, so there exist $m_1,m_2\in M\setminus \{1\}$ such that $m=m_1 m_2$. By (LF2) there exist $\alpha_1,\alpha_2$ in $\Cat$ such that $\alpha=\alpha_1\alpha_2$ and $\dmap(\alpha_i)=m_i$ for $i=1,\ldots,k$, and by (LF1) $\alpha_1$ and $\alpha_2$ are not invertible. This contradicts that $\alpha$ is an atom.
\end{proof}
The following lemma is an easy consequence of Lemma \ref{lemma_1}.
\begin{lemma}\label{lemma_3}
	Let $\Cat$ be a LCSC, let $M$ be an atomic and conical monoid, and let $\dmap:\Cat\to M$ be a length function satisfying (LF1) and (LF2). Then every $\alpha\in \Cat\setminus \Cat^{-1}$ can be written as a finite composition of atoms. 
\end{lemma}

\begin{remark}
	Under the hypothesis of Lemma \ref{lemma_3} the decomposition need not be unique. For example, let a length function $\dmap:\Cat\to M$ satisfying conditions (LF1) and (LF2) where $M=<2,5>_{\ZZ^+}\subsetneq \ZZ^+$. Then, by Lemma \ref{lemma_1}, we have that $\Lambda_a=\dmap^{-1}(2)\cup \dmap^{-1}(5)$. Thus, for any $\alpha\in \Cat$ with $\dmap(\alpha)=10$, there exist atomic factorizations $\alpha=\beta_1\cdots\beta_5=\gamma_1\gamma_2$ with $\dmap(\beta_i)=2$ for every $i=1,\ldots,5$ and $\dmap(\gamma_j)=5$ for every $j=1,2$. 
\end{remark}

Given a LCSC $\Cat$ and $\alpha\in \Cat$ we denote by $\alpha\Cat=\{\alpha\gamma: r(\gamma)=s(\alpha)\}$ a \emph{principal right ideal of $\Cat$}. From the above observations we can easily adapt the following results from \cite[Section 2]{LV2} to our context.

\begin{proposition}[{\cite[Lemma 2.3]{LV2}}]
	Let $\Cat$ be a LCSC, let $M$ be a conical atomic monoid, and let $\dmap:\Cat\to M$ be a length function satisfying (LF1) and (LF2). Then $\alpha\Cat$ is maximal if and only if $\alpha$ is an atom.
\end{proposition}

\begin{proposition}[{\cite[Proposition 2.6]{LV2}}]\label{prop_2.6}
 Let $\Cat$ be a LCSC, let $M$ be a conical atomic monoid, and let $\dmap:\Cat\to M$ be a length function satisfying (LF1) and (LF2), and let $B\subseteq \Cat_a$ be a transversal of generators of maximal principal right ideals. Then, $\Cat=B^*\Cat^{-1}$ where $B^*$ is the subcategory of $\Cat$ generated by $B$.
\end{proposition}

The next step is to realize such a $\Cat$ as a Zappa-Sz\'ep product $B^*\Join \Cat^{-1}$. In order to do that we need to impose an extra condition in the category $\Cat$: 

\begin{definition}\label{Def:R-cond}
Let $B\subseteq \Cat_a$ be a transversal of generators of maximal principal right ideals, then we say that $\Cat$ satisfies the \emph{$\mathcal{R}$-condition} if whenever $\alpha=\beta g$ for $\alpha,\beta\in B^*$ and $g\in \Cat^{-1}$ we have that $g\in \Cat^0$. 
\end{definition}

\begin{proposition}[{\cite[Theorem 4.2]{LV2}}]\label{thm_ZP}
Let $\Cat$ be a LCSC, let $M$ be a conical atomic monoid, and let $\dmap:\Cat\to M$ be a length function satisfying (LF1) and (LF2), let $B\subseteq \Cat_a$ be a transversal of generators of maximal principal right ideals, and suppose that $\Cat$ satisfies the $\mathcal{R}$-condition. Then, every element of $\Cat$ has a unique representation as an element of $B^*\Cat^{-1}$. Thus, $\Cat$ is the Zappa-Sz\'ep product $B^*\Join \Cat^{-1}$. Moreover, if $\Cat$ satisfies the WFP, then the restriction of $\dmap$ to $B^*$ satisfies the UFP. 
\end{proposition}

Now let $\Cat$ be a LCSC, let $M$ be a conical atomic monoid, let $\dmap:\Cat\to M$ be a length function satisfying (LF1) and (LF2), let $\Grpd$ be a discrete groupoid acting on $\Cat$, and let  $(\Cat,\dmap,\Grpd,\varphi)$ be a category system. We can naturally extend $\dmap$ to a length function $\dmap: \Cat \Join^\varphi \Grpd\to M$ by $\dmap(\alpha,g)=\dmap(\alpha)$ that trivially satisfies conditions $(LF1)$ and $(LF2)$. Also, if $\Cat$ satisfies the WFP then so does $\Cat\Join^\varphi \Grpd$ by Proposition \ref{ZP_WFP}. Moreover, since any transversal of generators of maximal principal right ideals of $\Cat$ is also a transversal of maximal principal right ideals of $\Cat\Join^\varphi \Grpd$, it is straightforward to check that   $\Cat\Join^\varphi \Grpd$ satisfies $\mathcal{R}$-condition whenever $\Cat$ does. Finally, observe that the action of the groupoid $\Grpd$ on $\Cat$ restricts to an action of $\Cat^{-1}$ (Remark \ref{Rem:ConditionsAction}(2)) and hence $(\Cat^{-1},\dmap,\Grpd,\varphi)$ is a category system too. So, we can conclude

\begin{corollary}\label{corol_ZP_WFP}
	Let $\Cat$ be a LCSC, let $M$ be a conical atomic monoid,  let $\dmap:\Cat\to M$ be a length function, let  $\Grpd$ be a discrete groupoid acting on $\Cat$, let  $(\Cat,\dmap,\Grpd,\varphi)$ be a category system, let $B\subseteq \Cat_a$ be a transversal of generators of maximal principal right ideals of $\Cat$, and suppose that $\Cat$ satisfies the $\mathcal{R}$-condition. Then, $\Cat\Join^\varphi \Grpd$ is the Zappa-Sz\'ep product $B^*\Join (\Cat^{-1}\Join^\varphi \Grpd)$. Moreover, if $\Cat$ satisfies the WFP, then the restriction of $\dmap$ to $B^*$ satisfies the UFP. 
\end{corollary}

\section{Properties of the tight groupoid associated to LCSC}

In this section we will extend the application of the results of \cite[Section 7 and 8]{OP_LCSC} to category systems.

First, we will recall the essential facts needed to understand what it follows.

\begin{definition}
	A semigroup $\Semi$ is an \emph{inverse semigroup} if for every $s\in\Semi$ there exists a unique $s^*\in\Semi$ such that $s=ss^*s$ and $s^*=s^*ss^*$.
\end{definition}
Equivalently, $\Semi$ is an inverse semigroup if and only if the subsemigroup $$\Idem:=\{e\in\Semi:e^2=e\}$$ of idempotents of $\Semi$ is commutative.

\begin{definition}
	Let $\Semi$ be an inverse semigroup, and let $\Idem$ be its subsemigroup of idempotents. Given $e,f\in\Idem$, we say that $e\leq f$ if and only if $e=ef$. We extend this relation to a partial order as follows: given $s,t\in\Semi$, we say that $s\leq t$ if and only if $s=ss^*t=ts^*s$.
\end{definition}

Given a LCSC $\Lambda$, we will define some inverse semigroups associated to $\Lambda$. First, we define the (symmetric) inverse semigroup on $\Cat$ as
$$\Inv(\Cat):=\{f:Y\to Z:Y,Z\subseteq X\text{ and }f \text{ is a bijection }\}\,,$$
endowed with operation 
$$g\circ f:f^{-1}(\ran(f)\cap \dom(g))\longrightarrow g(\ran(f)\cap \dom(g))\,,$$
and involution 
$$f^*:=f^{-1}:\ran(f)\longrightarrow \dom(f)\,.$$

Notice that $\Inv(\Cat)$ has unit $\text{Id}_\Cat:\Cat\to \Cat$ and zero being the empty map $0:\emptyset\to \emptyset$.

\begin{definition}
	Let $\Lambda$ be a LCSC. For any $\alpha\in\Lambda$, we define two elements of $\Inv(\Lambda)$:\begin{enumerate}
		\item  $\osh^\alpha:\alpha\Lambda\to s(\alpha)\Lambda$ given by $\alpha\beta \mapsto \beta\,,$
		\item  $\tau^\alpha:s(\alpha)\Lambda\to \alpha\Lambda$ given by $\beta \mapsto \alpha\beta\,.$
	\end{enumerate}
\end{definition}

\begin{definition}[{\cite[Lemma 2.3]{OP_LCSC}}]
	Given a finitely aligned  LCSC $\Lambda$, we define the inverse semigroup
	$$\Semi_\Cat:=\left\langle \osh^\alpha,\tau^\alpha:\alpha\in \Lambda\right\rangle \,.$$
\end{definition} 

Given $s\in \Semi_\Cat$ we denote by $\text{dom}(s)$ and $\text{ran(s)}$ the domain and the range of $s$ respectively.

\begin{definition}
	Let $\Lambda$ be a finitely aligned LCSC. A nonempty subset $F$ of $\Lambda$ is:\begin{enumerate}
		\item \emph{Hereditary}, if $\alpha\in\Lambda$, $\beta\in F$ and $\alpha\leq \beta$ implies $\alpha\in F$,
		\item \emph{(upwards) directed},  $\alpha,\beta\in F$ implies that there exists $\gamma\in F$ with $\alpha,\beta\leq \gamma$.
	\end{enumerate}
	We denote $\Lambda^*$ the set of nonempty, hereditary, directed subsets of $\Lambda$.
\end{definition}

\begin{definition} Given $\Cat$ a finitely aligned  LCSC, we say that $C\in\Lambda^*$ is \emph{maximal} if whenever $C\subset D$ with $D\in \Lambda^*$ we have that $D=\Lambda$.  We will denote $\Lambda^{**}:=\{C\in\Lambda^*:C\text{ is maximal}\}$.
\end{definition}

\begin{definition}\label{definition2_3_9}
	Let $\Cat$ be a finitely aligned  LCSC and $v\in\Cat^0$. We say that $C\in v\Cat^*$ is $\emph{tight}$ if  for every $\alpha\in C$  and every finite set $F$ of $\Lambda$ with $C\cap F=\emptyset$,  there exists $D\in\Cat^{**}$ with $\alpha\in D$ and $D\cap F=\emptyset$. We denote by $\Cat_{tight}$ the set of tight hereditary directed sets.
\end{definition}

For a finitely aligned  LCSC $\Cat$, in \cite[Defintion 3.21]{OP_LCSC} is defined a topology on $\Cat_{tight}$ that makes it a locally compact, Hausdorff and totally disconnected space. 
\begin{definition}
	Let $s=\elmap{\alpha}{\beta}\in \Semi_\Lambda$. Then, we define 
	$$E_\alpha:=\{C\in\Lambda^*: \alpha\in C \}\qquad\text{and}\qquad E_\beta:=\{C\in\Lambda^*: \beta\in C \}\,.$$
and a map $s:E_\beta \to E_\alpha$ given by 
$$s\cdot F=\bigcup_{\beta\leq \gamma,\,\gamma\in F}[\alpha\sigma^\beta (\gamma)]\,,$$
for every $F\in \Cat^*$. This action restricts to a continuous action on $\Cat_{tight}$ \cite[Corollary 4.7]{OP_LCSC}.
\end{definition}

Given $\Semi_\Cat\times \Cat_{tight}:=\{(s,F):F \cap \text{dom}(s)\neq \emptyset \}$, we define the following groupoid structure. Given $(s,F),(t,G)\in \Semi_\Cat\times \Cat_{tight}$:
\begin{enumerate}
	\item $d(s,F)=F$ and $r(s,F)=s\cdot F$,
	\item $(s,F)\cdot (t,G)$ is defined if $t\cdot G= F$, and then $(s,F)\cdot (t,G)=(st,G)$,
	\item $(s,F)^{-1}=(s^*,s\cdot F)$\,.
\end{enumerate} 
We say that $(s,F)\sim (t,G)$ if and only if $F=G$ and there exists $e\in \Idem[\Semi_\Cat]$ with $x\in \text{dom}(e)$ and $se=te$. This is an equivalence relation, compatible with the groupoid structure. Thus, we define \emph{the tight groupoid of $\Cat$} as $\G_{tight}(\Cat):=\Semi_\Cat\times \Cat_{tight}/\sim$, with the induced operations defined above, and the topology generated by the open sets of the form 
$$\Theta(s,U)=\{[s,F]:F\in U\cap \text{dom}(s)\}\,,$$
for $s\in \Semi_\Lambda$ and an open set $U\subseteq \Cat_{tight}$.
Then $\G_{tight}(\Cat)$ is   a locally compact 'etale groupoid. A nice description of $\G_{tight}(\Semi_\Lambda)$ is given in \cite[Lemma 4.8]{OP_LCSC}:
$$\G_{tight}(\Lambda)=\{[\elmap{\alpha}{\beta},F]: s(\alpha)=s(\beta),\, \beta\in F \}\,.$$

 Let $\Cat$ be a finitely aligned  LCSC with length function $\dmap:\Lambda\to\Gamma$, let $\Grpd$ be a discrete groupoid acting on $\Cat$, let $(\Cat,\dmap,\Grpd,\varphi)$ be a category system, and let $\Cat \Join^\varphi \Grpd$ be the associated Zappa- Sz\'ep product. Then observe that $\Cat$ can be identified with the full subcategory $\Cat \Join^\varphi \Grpdu$ of $\Cat \Join^\varphi \Grpd$, and hence $\Grpd_{tight}(\Cat)$ can be seen as an open subgroupoid of $\Grpd_{tight}(\Cat \Join^\varphi \Grpd)$.
  
 \begin{remark}\label{ZP_comp} Let $\Cat$ be a finitely aligned LCSC with length function $\dmap:\Lambda\to\Gamma$, let $\Grpd$ be a discrete groupoid acting on $\Cat$, let $(\Cat,\dmap,\Grpd,\varphi)$ be a category system, and let $\Cat \Join^\varphi \Grpd$ be the associated Zappa- Sz\'ep product.
 	\begin{enumerate}
 		\item Let $(\alpha,g),(\beta,f)\in \Cat \Join^\varphi \Grpd$,  and suppose that $(\alpha,g)\leq(\beta,f)$. Then,  there exists $(\delta,h)\in \Cat \Join^\varphi \Grpd $ with  $r(\delta,h)=s(\alpha,g)$, that is, $r(\delta)=g^{-1}\cdot s(\alpha)$ such that 
 	$$(\beta,f)=(\alpha,g)(\delta,h)=(\alpha(g\cdot \delta),\varphi(g,\delta)h)\,.$$
 	Hence,  $\alpha(g\cdot \delta)= \beta$ and $f=\varphi(g,\delta)h$. Thus, $g\cdot \delta=\sigma^\alpha(\beta)$ by left cancellation, and so $\delta=g^{-1}\cdot \sigma^\alpha(\beta)$ and $h=\varphi(g,g^{-1}\cdot\sigma^\alpha(\beta))^{-1}f=\varphi(g^{-1},\sigma^\alpha(\beta))f$ because of the cocycle identity. Therefore,  
 	$$(\alpha,g)\leq(\beta,f)\qquad\text{if and only if}\qquad \alpha\leq \beta\,,$$ and then we have that 
 	$$\sigma^{(\alpha,g)}(\beta,f)=(g^{-1}\cdot \sigma^\alpha(\beta),\varphi(g^{-1},\sigma^\alpha(\beta))f)\,.$$
 	The above observation, together with  Proposition \ref{proposition_LC and FA}(2)(a), shows that the map 
 	$$F\mapsto F\Join^\varphi \Grpd:=\{(\alpha,g)\in \Cat \Join^\varphi \Grpd:  \alpha\in F \} $$
 	is a bijection between $\Cat^*$ and $(\Cat \Join^\varphi \Grpd)^*$. This bijection clearly restricts to a bijection of their tight (maximal) hereditary upper-directed subsets, so we will identify $(\Cat \Join^\varphi \Grpd)_{tight}$ with $\Cat_{tight}$. 
 	\item  Let $(\alpha,g),(\beta,f)\in \Cat \Join^\varphi \Grpd$ with $s(\alpha,g)=s(\beta,f)$, and let $\elmap{(\alpha,g)}{(\beta,f)}\in \Semi_{\Cat \Join^\varphi \Grpd}$. Observe that, since  $(\beta,f)(\Cat \Join^\varphi \Grpd)=(\beta,s(f))(\Cat \Join^\varphi \Grpd)$, we have that $\elmap{(\alpha,g)}{(\beta,f)}=\elmap{(\alpha,gf^{-1})}{(\beta,s(f))}$. 
 	
 	\item Given $F\in \Cat_{tight}$ and $(\alpha,g),(\beta,f)\in \Cat \Join^\varphi \Grpd$ with $s(\alpha,g)=s(\beta,f)$, we have that 
 	$$\elmap{(\alpha,g)}{(\beta,f)}\cdot (F\Join^\varphi \Grpd)=F'\Join^\varphi \Grpd\,,$$
where $F':=\bigcup_{\beta\leq \gamma,\,\gamma\in F}[\alpha((gf^{-1})\cdot \sigma^\beta (\gamma))]\in \Cat_{tight}$.
 	 	\end{enumerate}
 \end{remark}

 First we will give necessary and sufficient conditions for $\Grpd_{tight}(\Cat \Join^\varphi \Grpd)$ to be a Hausdorff groupoid. 
 
 \begin{proposition}\label{ZP_hausddorff}
 Let $\Cat$ be a finitely aligned LCSC with length function $\dmap:\Lambda\to\Gamma$, let $\Grpd$ be a discrete groupoid acting on $\Cat$, let $(\Cat,\dmap,\Grpd,\varphi)$ be a category system, and let $\Cat \Join^\varphi \Grpd$ be the associated Zappa-Sz\'ep product.  Then, $\Grpd_{tight}(\Cat \Join^\varphi \Grpd)$ is a Hausdorff groupoid if and only if, given $\alpha,\beta\in \Cat$ and $g\in\Grpd$ with $g^{-1}\cdot s(\alpha)=s(\beta)$, there exists a finite subset  $H\subset s(\beta)\Cat$ such that 
 \begin{enumerate}
 \item $\alpha(g\cdot\gamma)=\beta\gamma$ and $\varphi(g,\gamma)=s(\gamma)$ for every $\gamma\in H$,
 \item for every $\delta\in s(\alpha)\Cat$ with $\alpha(g\cdot \delta)=\beta\delta$ and $\varphi(g,\delta)=s(\gamma)$, then $ \delta\Cap \gamma$ for some $\gamma\in H$.
 \end{enumerate}
\end{proposition}
\begin{proof}
According \cite[Theorem 3.16]{EP} the groupoid $\Grpd_{tight}(\Cat \Join^\varphi \Grpd)$ is Hausdorff if and only if for any $s\in \Semi_{\Cat \Join^\varphi \Grpd}$ the set
$$\mathcal{I}_s:=\{e\in \Idemp(\Cat \Join^\varphi \Grpd): e\leq s \}\,,$$
admits a finite cover, that is, a finite set $H\subseteq \mathcal{I}_s$ such that for every $e\in \mathcal{I}_s$ there exists $f\in H$ with $ef\neq 0$. Now, we will translate this equivalence to our language in terms of element of the category $\Cat$. By Remark \ref{ZP_comp}(2) we can assume without lost of generality that $s=\elmap{(\alpha,g)}{(\beta, s(g))}$. Observe that, given an idempotent of the form $e=\elmap{(\gamma,s(\gamma))}{(\gamma,s(\gamma))}$  for some $\gamma\in \Cat$, $es=se=s$ is equivalent to say that $\gamma=\alpha(g\cdot \sigma^\beta(\gamma))=\beta\sigma^\beta(\gamma)$ and $\varphi(g,\sigma^\beta(\gamma))=s(\gamma)$. Therefore, we can identify $\mathcal{I}_s$ with the set $$T_s:=\{\gamma\in r(\beta)\Cat: \alpha(g\cdot \gamma)=\beta\gamma\text{ and } \varphi(g,\gamma)=s(\gamma) \}\,.$$ Now, given $\gamma,\delta\in T_s$ we have that $$(\elmap{(\beta\gamma,s(\gamma))}{(\beta\gamma,s(\gamma))})(\elmap{(\beta\delta,s(\delta))}{(\beta\delta,s(\delta))} )\neq 0$$ 
if and only $\beta\gamma\Cap \beta \delta$, that by left cancellation is equivalent to $\gamma\Cap \delta$. Then, having a finite cover of $\mathcal{I}_s$ is equivalent to have a finite subset $H\subseteq T_s$ such that for every $\gamma\in T_s$ there exists $\delta\in H$  with $\gamma\Cap \delta$. 
\end{proof}

 \begin{corollary}\label{ZP_pseudofree_Hausdorff}
 	Let $\Cat$ be a finitely aligned LCSC with length function $\dmap:\Lambda\to\Gamma$ satisfying the WFP, let $\Grpd$ be a discrete groupoid acting on $\Cat$, let $(\Cat,\dmap,\Grpd,\varphi)$ be a category system, and let $\Cat \Join^\varphi \Grpd$ be the associated Zappa-Sz\'ep product. If $(\Cat,\dmap,\Grpd,\varphi)$ is pseudo-free, then $\G_{tight}(\Cat \Join^\varphi \Grpd)$ is a Hausdorff groupoid. 
 \end{corollary}  
\begin{proof}
Let $\alpha,\beta\in \Cat$ and $g^{-1}\cdot s(\alpha)=s(\beta)$, and suppose there exists $\gamma\in s(\beta)\Cat$ with $\alpha(g\cdot \gamma)=\beta\gamma$ and $\varphi(g,\gamma)=s(\gamma)$. Then, we have that 
$$\dmap(\alpha)\dmap(\gamma)=\dmap(\alpha)\dmap(g\cdot \gamma)=\dmap(\alpha(g\cdot \gamma))=\dmap(\beta)\dmap(\gamma)=\dmap(\beta)\dmap(\gamma)\,,$$
whence $\dmap(\alpha)=\dmap(\beta)$, and by the WFP there exists $f\in \Cat^{-1}$ such that $g\cdot \gamma=f\gamma$.  Since  $(\Cat,\dmap,\Grpd,\varphi)$ is pseudo-free, $g=s(\alpha)$, so $\alpha\gamma=\alpha(g\cdot \gamma)=\beta\gamma$. But by Lemma \ref{pseudofree_cat_right_cancellative} $\Cat$ is right cancellative, and so $\alpha=\beta$. Hence, taking $H=\{\alpha\}$ we can apply Proposition \ref{ZP_hausddorff}.
\end{proof}

Now, we are ready to characterize minimality and effectiveness for the corresponding tight groupoid. The following results are straightforward translations of the results in \cite[Section 6 \& 7]{OP_LCSC}.

\begin{proposition}\label{ZP_TP}
Let $\Cat$ be a finitely aligned LCSC with length function $\dmap:\Lambda\to\Gamma$, let $\Grpd$ be a discrete groupoid acting on $\Cat$, let $(\Cat,\dmap,\Grpd,\varphi)$ be a category system, and let $\Cat \Join^\varphi \Grpd$ be the associated Zappa-Sz\'ep product.  If $\G_{tight}(\Cat \Join^\varphi \Grpd)$ is Hausdorff or $\Cat_{tight}=\Cat^{**}$, then the following are equivalent:
\begin{enumerate}
\item $\G_{tight}(\Cat \Join^\varphi \Grpd)$ is topologically free,
\item Given $(\alpha,a),(\beta,b)\in \Cat \Join^\varphi \Grpd$ with $r(\alpha,a)=r(\beta,b)$ and $s(\alpha,a)=s(\beta,b)$, if $(\alpha,a)(\delta,d)\Cap (\beta,b)(\delta,d)$ for every $(\delta,d)\in s(\alpha,a)(\Cat \Join^\varphi \Grpd)$ then there exists $F\in\mathsf{FE}(s(\alpha,a))$ such that $(\alpha,a)(\gamma,d)=(\beta,b)(\gamma,d)$ for every $(\gamma,d)\in F$.
\item Given $\alpha,\beta\in \Cat$, $a,b\in \Grpd$ with $r(a)=s(\alpha)$, $r(b)=s(\beta)$, $r(\alpha)=r(\beta)$ and $a^{-1}\cdot s(\alpha)=b^{-1}\cdot s(\beta)$, if $\alpha (a\cdot \delta ) \Cap \beta(b\cdot \delta)$ for every $\delta\in (a^{-1}\cdot s(\alpha))\Lambda$ then there exists $F\in\mathsf{FE}(a^{-1}\cdot s(\alpha))$ such that $\alpha (a\cdot \gamma)=\beta(b\cdot\gamma)$ and $\varphi(a,\gamma)=\varphi(b,\gamma)$ for every $\gamma\in F$.
\end{enumerate}
\end{proposition}

\begin{proposition}\label{ZP_Min}
Let $\Cat$ be a finitely aligned  LCSC with length function $\dmap:\Lambda\to\Gamma$, let $\Grpd$ be a discrete groupoid acting on $\Cat$, let $(\Cat,\dmap,\Grpd,\varphi)$ be a category system, and let $\Cat \Join^\varphi \Grpd$ be the associated Zappa-Sz\'ep product.   If $\G_{tight}(\Cat \Join^\varphi \Grpd)$ is Hausdorff or $\Cat_{tight}=\Cat^{**}$, then the following are equivalent:
\begin{enumerate}
\item $\G_{tight}(\Cat \Join^\varphi \Grpd)$ is minimal.
\item For every $(\alpha,a),(\beta,b)\in \Lambda\Join^\varphi \Grpd$ there exists $F\in \mathsf{FE}(\alpha,a)$ such that for each $(\gamma,g)\in F$, $s(\beta,b)(\Lambda\Join^\varphi G) s(\gamma,g)\neq \emptyset$.
\item For every $\alpha,\beta\in \Lambda$ there exists $F\in \mathsf{FE}(\alpha)$ such that for each $\gamma\in F$,  there exists $g\in \Grpd$  with $s(\beta)\Lambda  (g\cdot s(\gamma))\neq \emptyset$.
\end{enumerate}
\end{proposition}

Thus, we conclude with the following result.

\begin{theorem}\label{theorem_simpleSystemnou}
Let $\Cat$ be a finitely aligned  LCSC with length function $\dmap:\Lambda\to\Gamma$, let $\Grpd$ be a discrete groupoid acting on $\Cat$. Let $(\Cat,\dmap,\Grpd,\varphi)$ be a category system and $\Cat \Join^\varphi \Grpd$ be the associated Zappa-Sz\'ep product.   If $\G_{tight}(\Cat \Join^\varphi \Grpd)$ is Hausdorff or $\Cat_{tight}=\Cat^{**}$, then the following are equivalent:
\begin{enumerate}
\item $C^*_r(\G_{tight}(\Cat \Join^\varphi \Grpd))$ (the reduced groupoid $C^*$-algebra) is simple.
\item For any field $K$, $K(\Lambda\Join^\varphi  G)$ (the Steinberg algebra) is simple.
\item The following properties hold:
\begin{enumerate}
\item Given $\alpha,\beta\in \Lambda$, $a,b\in \Grpd$ with $r(\alpha)=r(b)$ and $a^{-1}\cdot s(\alpha)=b^{-1}\cdot s(\beta)$, if $\alpha (a\cdot \delta ) \Cap \beta(b\cdot \delta)$ for every $\delta\in (a^{-1}\cdot s(\alpha))\Lambda$ then there exists $F\in\mathsf{FE}(a^{-1}\cdot s(\alpha))$ such that $\alpha (a\cdot \gamma)=\beta(b\cdot\gamma)$ and $\varphi(a,\gamma)=\varphi(b,\gamma)$ for every $\gamma\in F$.
\item For every $\alpha,\beta\in \Lambda$ there exists $F\in \mathsf{FE}(\alpha)$ such that for each $\gamma\in F$,  there exists $g\in \Grpd$  with $s(\beta)\Lambda  (g\cdot s(\gamma))\neq \emptyset$.
\end{enumerate}
\end{enumerate}
\end{theorem}

\section{Amenable groupoids}

Now, we will discuss the amenability of $\G_{tight}(\Cat \Join^\varphi \Grpd)$, following the same strategy used in \cite[Section 8]{OP_LCSC}.

Recall that $\Gamma$ is a submonoid of a group $Q$  with unit element $\ideQ$, and such that  $\Gamma\cap\Gamma^{-1}=\{\ideQ\}$. Then,  given two elements $\gamma_1,\gamma_2\in \Gamma$, 
$$\gamma_1\leq \gamma_2\qquad\text{if and only if}\qquad \gamma_1^{-1}\gamma_2 \in\Gamma\,.$$

\begin{lemma}[{c.f.\cite[Lemma 8.2]{OP_LCSC}}]\label{lemma_equality}
	Let $\Cat$ be a LCSC with length function $\dmap:\Lambda\to\Gamma$ satisfying the WFP. Let $\alpha,\beta\in \Lambda $ with $ \alpha\Cap \beta$. Then  $\alpha\leq \beta$ if and only if $\dmap(\alpha)\leq \dmap(\beta)$. In particular $\alpha\approx\beta$ whenever $\dmap(\alpha)=\dmap(\beta)$.
\end{lemma}

\begin{lemma}[{c.f.\cite[Lemma 8.5]{OP_LCSC}}]\label{ZP_cocyle}
	Let $\Cat$ be a finitely aligned  LCSC with length function $\dmap:\Lambda\to\Gamma$ satisfying the WFP, let $\Grpd$ be a discrete groupoid acting on $\Cat$, let $(\Cat,\dmap,\Grpd,\varphi)$ be a category system, and let $\Cat \Join^\varphi \Grpd$ be the associated Zappa- Sz\'ep product. Then, the map
	$$\dbar:\G_{tight}(\Cat \Join^\varphi \Grpd)\to Q\,,\qquad [\elmap{(\alpha,a)}{(\beta,b)},F \Join^\varphi \Grpd]\mapsto \dmap(\alpha)\dmap(\beta)^{-1}\,,$$
	is a well defined continuous groupoid homomorphism. In particular, $\dbar$ restricts to $\G_{tight}(\Cat)$.
\end{lemma}
 \begin{proof}
 	The  proof is verbatim that of \cite[Lemma 8.5]{OP_LCSC}, replacing \cite[Remark 8.4]{OP_LCSC} by Remark \ref{ZP_comp}(1) and \cite[Lemma 8.2]{OP_LCSC} by Lemma \ref{lemma_equality}, with the observation that, by the WFP, we have $\dmap^{-1}(\ideQ)=\Cat^{-1}$ (Remark \ref{rema_WFP}(1)).
 \end{proof}

Let $\dbar:\Grpd_{tight}(\Cat \Join^\varphi \Grpd)\to Q$ be the cocycle defined in Lemma \ref{ZP_cocyle}, and let us define
$$\Krpd(\Cat \Join^\varphi \Grpd):=(\dbar)^{-1}(\ideG)=\{[\elmap{(\alpha,g)}{(\beta,f)},F \Join^\varphi \Grpd]\in\G_{tight}(\Cat \Join^\varphi \Grpd) : \dmap(\alpha)\dmap(\beta)^{-1}=\ideG \}\,,$$ 
which is an open subgroupoid of $\G_{tight}(\Cat \Join^\varphi \Grpd)$.

In order to be able to decompose the groupoid $\Krpd(\Cat \Join^\varphi \Grpd)$ as a union of more treatable groupoids, we need to impose some conditions on the semigroup $\Gamma$.  

\begin{definition}
	Let $\Gamma\subseteq Q$ be a subsemigroup of a group $Q$ with $\Gamma\cap \Gamma^{-1}=\ideQ$. We say that $\Gamma$ is a \emph{join-semilattice} if given $g_1,g_2\in \Gamma$ 
	$$\inf\{g\in \Gamma: g_1,g_2\leq g\}$$
	exists and is unique. We will denote it by $g_1\vee g_2$.
\end{definition}

\begin{remark}\label{rema_joint}
Let $\Cat$ be a LCSC with length function $\dmap:\Lambda\to\Gamma$ satisfying the WFP, and assume that $\Gamma$ is a join-semilattice. Then, given $\alpha,\beta$ with $ \alpha\Cap \beta$, we have that $\dmap(\gamma)=\dmap(\alpha)\vee \dmap(\beta)$ for every $\gamma\in \alpha\vee \beta$. Indeed, let $\gamma\in \alpha\vee \beta$. Since $\alpha,\beta\leq \gamma$, then $\dmap(\alpha),\dmap(\beta)\leq \dmap (\gamma)$, so $\dmap(\alpha)\vee \dmap(\beta)\leq \dmap(\gamma)$.  By the WFP, there exist $\gamma_1,\gamma_2\in \Cat$ such that $\gamma=\gamma_1\gamma_2$ with $\dmap(\gamma_1)=\dmap(\alpha)\vee \dmap(\beta)$. Now,  $\dmap(\alpha),\dmap(\beta)\leq \dmap(\gamma_1)$, and so we have that $\alpha,\beta\leq \gamma_1\leq \gamma$ by Lemma \ref{lemma_equality}. But since $\gamma$ is a minimal extension of $\alpha$ and $\beta$, it follows that $\gamma\approx \gamma_1$, so $\dmap(\gamma)=\dmap(\gamma_1)=\dmap(\alpha)\vee \dmap(\beta)$. In particular we have that $\Cat$ is singly aligned. 
\end{remark}

We now assume that $\Gamma$ is a join-semilattice. Then, given $g\in \Gamma$, we define
$$\Krpd_g(\Cat \Join^\varphi \Grpd):=\{[\elmap{(\alpha,g)}{(\beta,f)},F \Join^\varphi \Grpd]: \dmap(\alpha)=\dmap(\beta)\leq g\}\,.$$
We claim that $\Krpd_g(\Cat \Join^\varphi \Grpd)$ is an open subgroupoid of $\Krpd(\Cat \Join^\varphi \Grpd)$. Let  $[\elmap{(\alpha,g)}{(\beta,f)},F \Join^\varphi \Grpd]$, $ [\elmap{(\delta,h)}{(\eta,l)},F' \Join^\varphi \Grpd]\in \Krpd_g(\Cat \Join^\varphi \Grpd)$ two composable elements with $g_1:=\dmap(\beta)=\dmap(\alpha)\leq g$ and $g_2:=\dmap(\delta)=\dmap(\eta)\leq g$. Since $\beta,\delta\in F$, there exists $\varepsilon\in (\beta\vee \delta)\cap F$, and $\dmap(\varepsilon)=\dmap(\alpha)\vee \dmap(\beta)=g_1\vee g_2 \leq g$ (Remark \ref{rema_joint}).  Then, 
\begin{align*}
[\elmap{(\alpha,g)}{(\beta,f)},F] & =[\elmap{(\alpha,gf^{-1})}{(\beta,s(f))} ,F] \\ & =[\elmap{(\alpha,gf^{-1})}{(\beta,s(f))}\elmap{(\varepsilon,s(\varepsilon))}{(\varepsilon,s(\varepsilon))} ,F] \\
 & =[\elmap{(\alpha (gf^{-1}\cdot \sigma^\beta(\varepsilon) ),\varphi(\sigma^\beta(\varepsilon),gf^{-1}))}{(\varepsilon,s(f))} ,F]
\end{align*}
and 
\begin{align*}
[\elmap{(\delta,h)}{(\eta,l)},F'] & =[\elmap{(\delta,r(h))}{(\eta,lh^{-1})} ,F'] \\ & =[\elmap{(\varepsilon,s(\varepsilon))}{(\varepsilon,s(\varepsilon))}\elmap{(\delta,r(h))}{(\eta,lh^{-1})} ,F'] \\
 & =[\elmap{(\varepsilon,s(\varepsilon))}{ (\sigma^\delta(\varepsilon)\eta, lh^{-1}) } ,F]
\end{align*}
while the product give us
$$[\elmap{(\alpha (gf^{-1}\cdot \sigma^\beta(\varepsilon) ),\varphi(\sigma^\beta(\varepsilon),gf^{-1}))}{(\varepsilon,s(f))} ,F][\elmap{(\varepsilon,s(\varepsilon))}{ (\sigma^\delta(\varepsilon)\eta, lh^{-1}) } ,F]=$$
$$=[\elmap{(\alpha (gf^{-1}\cdot \sigma^\beta(\varepsilon) ),\varphi(\sigma^\beta(\varepsilon),gf^{-1}))}{(\sigma^\delta(\varepsilon)\eta,lh^{-1})} ,F']$$ 
But 
\begin{align*}
\dmap(\alpha (gf^{-1}\cdot \sigma^\beta(\varepsilon))) & =\dmap(\alpha) \dmap(gf^{-1}\cdot \sigma^\beta(\varepsilon))=\dmap(\alpha) \dmap(\sigma^\beta(\varepsilon) \\
 & =\dmap(\alpha)\dmap(\beta)^{-1}\dmap(\varepsilon)=\dmap(\alpha)\dmap(\alpha)^{-1}\dmap(\varepsilon)=\dmap(\varepsilon)=g_1\vee g_2\,.
\end{align*}
and 
$$\dmap(\sigma^\delta(\varepsilon)\eta)=\dmap(\varepsilon)\dmap(\delta)^{-1}\dmap(\eta)=\dmap(\varepsilon)\dmap(\eta)^{-1}\dmap(\eta)=\dmap(\varepsilon)=g_1\vee g_2\,.$$
Therefore,   $[\elmap{(\alpha,g)}{(\beta,f)},F][\elmap{(\delta,h)}{(\eta,l)},F']\in \Krpd_{g_1\vee g_2}(\Cat \Join^\varphi \Grpd)\subseteq \Krpd_{g}(\Cat \Join^\varphi \Grpd)$, as desired.

Moreover, as a consequence of the above computation, given $g_1,g_2\leq g$ we have that $\Krpd_{g_1}(\Cat \Join^\varphi \Grpd)\Krpd_{ g_2}(\Cat \Join^\varphi \Grpd)\subseteq \Krpd_{g}(\Cat \Join^\varphi \Grpd)$.  Then, if $\Gamma$ is countable, there exists an ascending sequence of elements $g_1,g_2,\ldots \in \Gamma$ such that for every $g\in \Gamma$ there exists $n\in \NN$ with $g\leq g_n$. Hence,   $\Krpd (\Cat \Join^\varphi \Grpd)=\bigcup_{i=1}^{\infty}\Krpd_{g_i} (\Cat \Join^\varphi \Grpd)$.
\vspace{.2truecm}

The next step will be to define a cocycle of the groupoids $\Krpd_{g}(\Cat \Join^\varphi \Grpd)$ onto $\Cat^{-1}\Join^\varphi\Grpd$. In order to do that, we will need to make the following assumption.

\begin{definition}\label{def_bigstar}
	Let $\Cat$ be a LCSC with length function $\dmap:\Lambda\to\Gamma$ satisfying the WFP. We will say that  $\Cat$ satisfies property $(\bigstar)$ if given $F\in \Cat_{tight}$ and $g\in \Gamma$, then there exists $\beta_g\in F$ (non-necessarily unique)  with $\dmap(\beta_g)\leq g$ such that, whenever $\alpha\in F$ satisfies $\dmap(\alpha)\leq g$,  we have that $\alpha\leq \beta$. 
\end{definition}

\begin{remark}\label{rema_mark_elements}Let $\Cat$ be a LCSC with length function $\dmap:\Lambda\to\Gamma$ satisfying the WFP and property $(\bigstar)$.
	\begin{enumerate}
		\item Let $F\in \Cat_{tight}$, and let $g\in \Gamma$ and let $\beta_g\in \Cat$ satisfying the conditions of Definition \ref{def_bigstar}. Suppose that there exists another $\beta'_g$ satisfying the same properties than $\beta_g$. Then, $\dmap(\beta_g)=\dmap(\beta'_g)$, and hence by Lemma \ref{lemma_equality} we have that $\beta_g\approx \beta'_g$. Thus, $\beta_g$ is unique up to an invertible. 
		
		\item  Let $F\in \Cat_{tight}$, and let $g\in \Gamma$ and let $\beta_g\in \Cat$ satisfying the conditions of Definition \ref{def_bigstar}. Now, let $\alpha\in \Cat$ such that $s(\beta_g)=s(\alpha)$ and $\dmap(\beta_g)=\dmap(\alpha)$. Recall that 
		$$F':=\elmap{\alpha}{\beta_g}\cdot F=\bigcup_{\beta_g\leq \gamma,\,\gamma\in F}[\alpha\sigma^{\beta_g} (\gamma)]\,.$$
		Let $\beta'_g\in F'$ satisfying the conditions of Definition \ref{def_bigstar}. Then, $\alpha\leq \beta'_g$. 
		On the other hand, there exists $\gamma\in F$ such that $\alpha\sigma^\alpha(\beta'_g)=\beta'_g\leq \alpha\sigma^{\beta_g}(\gamma)$, whence by left cancellation $\sigma^\alpha(\beta'_g)\leq \sigma^{\beta_g}(\gamma)$, and $\beta_g\sigma^\alpha(\beta'_g)\leq \beta_g \sigma^{\beta_g}(\gamma)=\gamma$. Thus, $\beta_g\sigma^\alpha(\beta'_g)\in F$ and $\dmap(\beta_g\sigma^\alpha(\beta'_g))=\dmap(\beta'_g)\leq g$. Then, by property  $(\bigstar)$ we have that $\beta_g\sigma^\alpha(\beta'_g)\leq \beta_g$, whence by left cancellation $\sigma^\alpha(\beta'_g)$ must be an invertible. Thus, $\beta'_g\approx \alpha$.
	\end{enumerate}	
\end{remark}

\begin{proposition}[{c.f.\cite[Lemma 8.8]{OP_LCSC}}]\label{Star-property}
	Let $\Cat$ be a LCSC with length function $\dmap:\Lambda\to\Gamma$ satisfying the WFP, and assume every bounded ascending sequence of elements of $\Gamma$ stabilizes. Then  $\Lambda$ satisfies property $(\bigstar)$.
\end{proposition}

Given a LCSC $\Cat$, a discrete groupoid $\Grpd$ acting on $\Cat$ and a category system $(\Cat,\dmap,\Grpd,\varphi)$, by Lemma \ref{lemma_inv} we have that $(\Cat \Join^\varphi \Grpd)^{-1}=\Cat^{-1} \Join^\varphi \Grpd$. Thus, $\Cat^{-1} \Join^\varphi \Grpd$ is a  discrete groupoid.

 \begin{proposition}\label{G_cocycle}
	Let $\Cat$ be a LCSC with length function $\dmap:\Lambda\to\Gamma$ satisfying the WFP, let $\Grpd$ be a discrete groupoid acting on $\Cat$, let $(\Cat,\dmap,\Grpd,\varphi)$ be a pseudo-free category system, let $\Cat \Join^\varphi \Grpd$ be the associated Zappa-Sz\'ep product, and suppose that $\Lambda$ satisfies property $(\bigstar)$. Then, for every $g\in\Gamma$, there exists a continuous groupoid homomorphism
	$$\tmap^{(g)}:\Krpd_g(\Cat \Join^\varphi \Grpd)\to \Cat^{-1} \Join^\varphi \Grpd\,.$$
\end{proposition}
\begin{proof}
Fix $g\in \Gamma$, and for every $F\in \Cat_{tight}$ choose $\beta_{F}$ satisfying property $(\bigstar)$. Let $[\elmap{(\alpha,f)}{(\gamma,s(f))},F]\in \Krpd_g(\Cat \Join^\varphi \Grpd)$. Then, $\dmap(\alpha)=\dmap(\gamma)\leq g$, and hence $\gamma\leq \beta_F$. Therefore,
\begin{align*}
[\elmap{(\alpha,f)}{(\gamma,s(f))},F] & =[\elmap{(\alpha,f)}{(\gamma,s(f))}\elmap{(\beta_F,s(\beta_F))}{(\beta_F,s(\beta_F))},F] \\
& =[\elmap{(\alpha (f\cdot \sigma^\gamma(\beta_F)),\varphi(f,\sigma^\gamma(\beta_F) ))}{(\beta_F,s(\beta_F))},F]\,.
\end{align*}
By Remark \ref{rema_mark_elements} we have that $\alpha (f\cdot \sigma^\gamma(\beta_F))=\beta_{F'} \xi$, where $F':=\elmap{(\alpha,f)}{(\gamma,s(f))}\cdot F$ and $\xi\in \Cat^{-1}$. So, 
$$[\elmap{(\alpha,f)}{(\gamma,s(f))},F]=[\elmap{(\beta_{F'}\xi,\varphi(f,\sigma^\gamma(\beta_F) ))}{(\beta_F,s(\beta_F))},F]\,.$$
Therefore, every element in $\Krpd_g(\Cat \Join^\varphi \Grpd)$ is of the form $[\elmap{(\beta_{F'}\xi,h)}{(\beta_F,s(\beta_F))},F]$ for some $\xi\in \Cat^{-1}$ and $h\in \Grpd$, and then we define
$$\tmap^{(g)}([\elmap{(\beta_{F'}\xi,h)}{(\beta_F,s(\beta_F))},F])=(\xi,h)\,.$$

	 Let us check that $\tmap^{(g)}$ is well defined. To this end, let $[\elmap{(\beta_{F'}\xi_1,h_1)}{(\beta_F,s(\beta_F))},F]$ and $[\elmap{(\beta_{F'}\xi_2,h_2)}{(\beta_F,s(\beta_F))},F]$ in $\Krpd_g(\Cat \Join^\varphi \Grpd)$ with $$[\elmap{(\beta_{F'}\xi_1,h_1)}{(\beta_F,s(\beta_F))},F]=[\elmap{(\beta_{F'}\xi_2,h_2)}{(\beta_F,s(\beta_F))},F]\,.$$ Then there exists $\gamma\in F$ with $\beta_F\leq \gamma$ such that 
	 $$(\elmap{(\beta_{F'}\xi_1,h_1)}{(\beta_F,s(\beta_F))})(\elmap{(\gamma,s(\gamma))}{(\gamma,s(\gamma))})$$
	 $$=\elmap{(\beta_{F'}\xi_1(h_1\cdot \sigma^{\beta_F}(\gamma)),\varphi(h_1,\sigma^{\beta_F}(\gamma) ))}{(\gamma,s(\gamma))}$$
	 $$=(\elmap{(\beta_{F'}\xi_2,h_2)}{(\beta_F,s(\beta_F))})(\elmap{(\gamma,s(\gamma))}{(\gamma,s(\gamma))})$$
	 $$=\elmap{(\beta_{F'}\xi_2(h_2\cdot \sigma^{\beta_F}(\gamma)),\varphi(h_2,\sigma^{\beta_F}(\gamma) ))}{(\gamma,s(\gamma))}\,.$$
	 Therefore,  
	 $$\beta_{F'}\xi_1(h_1\cdot \sigma^{\beta_F}(\gamma))=\beta_{F'}\xi_2(h_2\cdot \sigma^{\beta_F}(\gamma))\qquad\text{and}\qquad \varphi(h_1,\sigma^{\beta_F}(\gamma))=\varphi(h_2,\sigma^{\beta_F}(\gamma))\,.$$
	 In particular, by left cancellation we have that 
	 $$\xi_1(h_1\cdot \sigma^{\beta_F}(\gamma))=\xi_2(h_2\cdot \sigma^{\beta_F}(\gamma))\,,$$ 
	 and 
	 $$\xi_2^{-1}\xi_1(h_1\cdot \sigma^{\beta_F}(\gamma))=h_2\cdot \sigma^{\beta_F}(\gamma)\,,$$
	 $$h^{-1}_2\cdot(\xi_2^{-1}\xi_1(h_1\cdot \sigma^{\beta_F}(\gamma)))=\sigma^{\beta_F}(\gamma)\,,$$
	 $$(h^{-1}_2\cdot\xi_2^{-1}\xi_1)(\varphi(h^{-1}_2,\xi_2^{-1}\xi_1)h_1\cdot \sigma^{\beta_F}(\gamma)))=\sigma^{\beta_F}(\gamma)\,,$$
	 $$\varphi(h^{-1}_2,\xi_2^{-1}\xi_1)h_1\cdot \sigma^{\beta_F}(\gamma))=(h^{-1}_2\cdot\xi_2^{-1}\xi_1)^{-1}\sigma^{\beta_F}(\gamma)\,.$$
	 On the other hand, we have that
	 \begin{align*}
	 s(\gamma) & = \varphi(h_2,\sigma^{\beta_F}(\gamma))^{-1}\varphi(h_1,\sigma^{\beta_F}(\gamma)) = \varphi(h^{-1}_2,h_2\cdot \sigma^{\beta_F}(\gamma))\varphi(h_1,\sigma^{\beta_F}(\gamma)) \\
	 & = \varphi(h^{-1}_2,\xi_2^{-1}\xi_1(h_1\cdot \sigma^{\beta_F}(\gamma)))\varphi(h_1,\sigma^{\beta_F}(\gamma)) \\
	 	 & = \varphi(\varphi(h^{-1}_2,\xi_2^{-1}\xi_1), h_1\cdot \sigma^{\beta_F}(\gamma))\varphi(h_1,\sigma^{\beta_F}(\gamma)) \\
	 	 	 	 & = \varphi(\varphi(h^{-1}_2,\xi_2^{-1}\xi_1) h_1, \sigma^{\beta_F}(\gamma))\,.
	 \end{align*}
	 Since the system is pseudo-free, we have that 
	 $$\varphi(h^{-1}_2,\xi_2^{-1}\xi_1)h_1=s(h_1)\qquad\text{and}\qquad (h^{-1}_2\cdot\xi_2^{-1}\xi_1)^{-1}=s(\beta_F)\,.$$
	 From the second equality we deduce that $\xi_2^{-1}\xi_1=s(\beta_F)$, and hence $\xi_1=\xi_2$. Thus, 
	 $$\varphi(h^{-1}_2,\xi_2^{-1}\xi_1)h_1=\varphi(h^{-1}_2,s(\xi_1)h_1=h^{-1}_2h_1=s(h_1)\,,$$
	 and so $h_2=h_1$. Hence, $\tmap^{(g)}$ is a well-defined map, as desired.

	Now, two composable elements $(x,y)\in \Krpd_g(\Cat \Join^\varphi \Grpd)^{(2)}$ are of the form  $$x=[\elmap{(\beta_{F''}\xi_1,h_1)}{(\beta_{F'},s(\beta_{F'}))},F']\qquad\text{and}\qquad y=[\elmap{(\beta_{F'}\xi_2,h_2)}{(\beta_I,s(\beta_F))},F]\,,$$
	where $\xi_1,\xi_2\in \Cat^{-1}$, $h_1,h_2\in \Grpd$, $F\in \Cat_{tight}$, $F':=\elmap{(\beta_{F'}\xi_2,h_2)}{(\beta_F,s(\beta_F))}\cdot F$ and $F'':=\elmap{(\beta_{F''}\xi_1,h_1)}{(\beta_{F'},s(\beta_{F'}))}\cdot F'$. Since 
	$$xy=[\elmap{(\xi_1(h_1\cdot \xi_2),\varphi(h_1,\xi_2)h_2)}{(\beta_F,s(\beta_F))},F]\,,$$
	we have that 
	$$\tmap^{(g)}(xy)=(\xi_1(h_1\cdot \xi_2),\varphi(h_1,\xi_2)h_2)=(\xi_1,h_1)(\xi_2,h_2)=\tmap^{(g)}(x)\tmap^{(g)}(y)\,.$$
	Thus, $\tmap^{(g)}$ is a groupoid homomorphism.	  
	  
	  Finally, given any $x=(\xi,h)\in \Cat^{-1} \Join^\varphi \Grpd$ we have that 
	  \begin{align*}
	  (\tmap^{(g)}(\xi,h))^{-1}(x) & = \{[\elmap{(\beta_{F'}\xi,h)}{(\beta_F,s(\beta_F))},F]: F\in \Cat_{tight}\} \\ 
	  & = \bigcup_{F\in \Cat_{tight},\,s(\beta_F)=s(h)}  \{[\elmap{(\beta_{F'}\xi,h)}{(\beta_F,s(\beta_F))},G]: G\in \Cat_{tight} \text{ and }\beta_F\in G\}\,, 
	  \end{align*}
	  that is a union of open sets of $\Krpd_g(\Cat \Join^\varphi \Grpd)$, so $\tmap^{(g)}$ is continuous.

\end{proof}

Given two groupoids $\Grpd$ and $\HG$ a groupoid homomorphism $p:\Grpd\to \HG$ is called \emph{strongly surjective} if $p$ is surjective and for every $x\in \Grpd^{(0)}$ we have that $p(r^{-1}(x))=r^{-1}(p(x))$ (see \cite[Definition 5.2.7]{AR}).

In general the map $\tmap^{(g)}:\Krpd_g(\Cat \Join^\varphi \Grpd)\to \Cat^{-1} \Join^\varphi \Grpd$ defined in Proposition \ref{G_cocycle} is not strongly surjective. However, we can consider the subgroupoid 
$$\HG_g:=\tmap^{(g)}(\Krpd_g(\Cat \Join^\varphi \Grpd))=$$
$$\{(\xi,g)\in \Cat^{-1} \Join^\varphi \Grpd: \exists F,F'\in \Cat_{tight},\, s(\beta_F)=d(\phi(g)),\, r(\xi)=s(\beta_{F'})   \}\,.$$ 

Let $F'\in \Krpd_g(\Cat \Join^\varphi \Grpd)^{(0)}$  (here we identify $\Krpd_g(\Cat \Join^\varphi \Grpd)^{(0)}$ with $\Cat_{tight}$), then 
$$
\tmap^{(g)}(r^{-1}(F'))=$$
$$\tmap^{(g)}(\{[\elmap{(\beta_{F'}\xi,h)}{(\beta_F,s(\beta_F))},F] :   r(\xi)=s(\beta_{F'}),\,\exists F\in \Cat_{tight} \text{ with } s(\beta_F)=s(g) \} )=$$ 
$$ \{(\xi,g)\in \Cat^{-1} \Join^\varphi \Grpd:  r(\xi)=s(\beta_{F'}),\,\exists F\in \Cat_{tight} \text{ with } s(\beta_F)=s(g)  \}
$$

and 
$$r^{-1}(\tmap^{(g)}(F'))\cap \HG_g=r^{-1}(s(\beta_{F'}),s(\beta_{F'}))\cap \HG_g=\{(\xi,g)\in \Cat^{-1} \Join^\varphi \Grpd: r(\xi)=s(\beta_{F'})\} \cap \HG_g$$
$$=\{(\xi,g)\in \Cat^{-1} \Join^\varphi \Grpd: r(\xi)=s(\beta_{F'}),\,\exists F\in \Cat_{tight}  \text{ with } s(\beta_F)=s(g)\}\,.$$
Therefore the map  $\tmap^{(g)}:\Krpd_g(\Cat \Join^\varphi \Grpd)\to \HG_g$ is strongly surjective.

 \begin{proposition}\label{amenability}
	Let $\Cat$ be a LCSC with length function $\dmap:\Lambda\to\Gamma$ satisfying the WFP, let $\Grpd$ be a discrete groupoid acting on $\Cat$, let $(\Cat,\dmap,\Grpd,\varphi)$ be a pseudo-free category system, let $\Cat \Join^\varphi \Grpd$ be the associated Zappa-Sz\'ep product, and suppose that $\Lambda$ satisfies property $(\bigstar)$. Suppose that the groupoid $\Cat^{-1} \Join^\varphi \Grpd$ and the group $Q$ are amenable. Moreover, assume that $\Gamma$ is a join-semilattice. Then the following statements are equivalent:
	\begin{enumerate}
		\item $\Grpd_{tight}(\Cat \Join^\varphi \Grpd)$ is amenable,
		\item the kernel of the map $\overline{\dmap}:\Grpd_{tight}(\Cat)\to Q$ is amenable,
		\item $\Grpd_{tight}(\Cat)$ is amenable.
		\end{enumerate}
\end{proposition}
\begin{proof}
$(2)$ and $(3)$ are equivalent due to 	\cite[Corollary 4.5]{RW} and that $\ker \overline{\dmap}$ is an open subgroupoid of $\Grpd_{tight}(\Cat)$. That $(1)$ implies $(3)$ is because $\Grpd_{tight}(\Cat)$ is an open subgroupoid of $\Grpd_{tight}(\Cat \Join^\varphi \Grpd)$. So, it is enough to prove that $(2)$ implies $(1)$.

 Let $\bar{\tmap}:\Grpd_{tight}(\Cat \Join^\varphi \Grpd)\to Q$ be the cocycle defined in Lemma \ref{ZP_cocyle}. By \cite[Corollary 4.5]{RW}, it is enough to prove that $\bar{\dmap}^{-1}(\ideQ)= \Krpd(\Cat \Join^\varphi \Grpd)$ is amenable.  Moreover, since $\Krpd(\Cat \Join^\varphi \Grpd)=\bigcup_{g\in \Gamma} \Krpd_g(\Cat \Join^\varphi \Grpd)$, we have that $\Krpd(\Cat \Join^\varphi \Grpd)$ is amenable if $\Krpd_g(\Cat \Join^\varphi \Grpd)$ is amenable for every $g\in \Gamma$ \cite[Section 5.2(c)]{AR}. Now, let $\tmap^{(g)}:\Krpd_g(\Cat \Join^\varphi \Grpd)\to \Cat^{-1} \Join^\varphi \Grpd$ defined in Proposition \ref{G_cocycle}. By assumption $\Cat^{-1} \Join^\varphi \Grpd$ is amenable, and thgus $\HG_g:=\tmap^{(g)}(\Krpd_g(\Cat \Join^\varphi \Grpd))\subseteq \Cat^{-1} \Join^\varphi \Grpd$ is amenable and the map $\tmap^{(g)}:\Krpd_g(\Cat \Join^\varphi \Grpd)\to \HG_g$ is strongly surjective (see the comments after Proposition \ref{G_cocycle}). Then, by \cite[Theorem 5.2.14]{AR}, $\Krpd_g(\Cat \Join^\varphi \Grpd)$ is amenable if $\ker \tmap^{(g)}$ is amenable. But observe that, if we identify $\Grpd_{tight}(\Cat)$ with $\Grpd_{tight}(\Cat \Join^\varphi \Grpdu)$ , then we have that 
$$\ker \tmap^{(g)}=\Krpd_g(\Cat \Join^\varphi \Grpd)\cap \Grpd_{tight}(\Cat)\subseteq \ker \overline{\dmap}\cap \Grpd_{tight}(\Cat)\,,$$ so  $\ker \tmap^{(g)}$ is an open subgroupoid of $\ker \overline{\dmap}$. Thus, $\ker \tmap^{(g)}$ is amenable by hypothesis. 
\end{proof}

 \begin{corollary}\label{corol_amenability}
	Let $\Cat$ be a LCSC with length function $\dmap:\Lambda\to\Gamma$ satisfying the UFP, let $\Grpd$ be a discrete groupoid acting on $\Cat$, let $(\Cat,\dmap,\Grpd,\varphi)$ be a pseudo-free category system, let $\Cat \Join^\varphi \Grpd$ be the associated Zappa-Sz\'ep product, and suppose that $\Lambda$ satisfies property $(\bigstar)$. Moreover, assume that $\Gamma$ is a join-semilattice. Suppose that the groupoid $\Grpd$ and the group $Q$ are amenable. Then $\Grpd_{tight}(\Cat \Join^\varphi \Grpd)$ is an amenable groupoid. 
\end{corollary}
\begin{proof}
First observe that, since $\dmap:\Lambda\to\Gamma$ has the UFP, $\Cat$ has no inverses. So, $\Cat^{-1} \Join^\varphi \Grpd$ is isomorphic to $\Grpd$. Then, by Proposition \ref{amenability}, it is enough to prove that the kernel of the map $\overline{\dmap}:\Grpd_{tight}(\Cat)\to Q$ is amenable. But this was shown in \cite[Section 8]{OP_LCSC}.
\end{proof}

\begin{example}
	Let $\Cat$ be the category with objects $\{v_i\}_{i\in \ZZ}$ and freely generated by the morphisms $\{\alpha_i\}_{i\in \ZZ}$, $\{\gamma_i\}_{i\in \ZZ}$ and $\{\gamma^{-1}_i\}_{i\in \ZZ}$, such that 
	$$s(\alpha_i)=v_i\qquad \text{and} \qquad r(\alpha_i)=v_{i+1}\,,$$
	$$s(\gamma_i)=v_i\qquad \text{and} \qquad r(\gamma_i)=v_{i}\,,$$
		$$s(\gamma^{-1}_i)=v_i\qquad \text{and} \qquad r(\gamma^{-1}_i)=v_{i}\,,$$
		for every $i\in \ZZ$, modulo the relation 
		$$\gamma_i\gamma^{-1}_i=\gamma_i^{-1}\gamma_i=v_i\,,$$
		for every $i\in \ZZ$. $\Cat$ is a left and right cancellative small category. We can define the length function $\dmap:\Cat\to \NN$ given by $\dmap(\alpha_i)=1$ and $\dmap(\gamma_i)=\dmap(\gamma_i^{-1})=0$ for every $i\in \ZZ$. Then $\dmap$ satisfies the WFP and $\Cat^{-1}$ is generated by $\{\gamma_i,\gamma_i^{-1}\}_{i\in \ZZ}$. Now we define the groupoid $\Grpd$ with unit space $\Grpdu=\{v_i\}_{i\in \ZZ}$ and generated by $\{g_i\}_{i\in \ZZ}$ with $s(g_i)=v_i$ and $r(g_i)=v_{i-1}$. We define $\phi:\Grpd\to \PisoG$ such that $\phi(g_i):v_i\Cat \to v_{i-1}\Cat$ is such that it translates the path one position to the left. Observe that then 
		$$\phi(g_i)_{|\gamma_i}=\phi(g_i)\qquad \text{and}\qquad \phi(g_i)_{|\alpha_{i-1}}=\phi(g_{i-1})\,$$
		for every $i\in \ZZ$. Therefore, the action of $\Grpd$ on $\Cat$ is self-similar, and we define the cocyle $\varphi:\Grpd {}_d\times_r \Cat\to \Grpd$ by $\varphi(g_i,\gamma_i):=g_i$ and $\varphi(g_{i},\alpha_{i-1}):=g_{i-1}$. Observe also that $(\Cat,\dmap,\Grpd,\varphi)$ is a pseudo-free category system, and hence $\Grpd_{tight}(\Cat \Join^\varphi \Grpd)$ is a Hausdorff groupoid by Corollary \ref{ZP_pseudofree_Hausdorff}.

		We will check that the groupoid $\Grpd_{tight}(\Cat \Join^\varphi \Grpd)$ is topologically free and minimal. Let $\beta_1, \beta_2\in \Cat$ with $r(\beta_1)=r(\beta_2)$, let $a_1,a_2\in \Grpd$ such that $r(a_1)=s(\beta_1)=v_k$, $r(a_2)=s(\beta_2)=v_l$ and $s(a_1)=s(a_2)=v_j$, and assume that $\beta_1(a_1\cdot \delta)\Cap \beta_2(a_2\cdot \delta)$ for every $\delta\in v_j\Cat$. We claim that then $\beta_1=\beta_2$. Indeed, first observe that $\beta_1(a_1\cdot v_j)\Cap \beta_2(a_2\cdot v_j)$, so $\beta_1\Cap \beta_2$. But this mean that either  $\beta_1\leq \beta_2$ or $\beta_2\leq \beta_1$.  Suppose that $\beta_1\leq \beta_2$, so that there exists $\gamma\in v_k\Cat$ such that $\beta_1\gamma=\beta_2$. But, by hypothesis, $\beta_1 (a_1\cdot \delta)\Cap \beta_1\gamma (a_2\cdot \delta)$ for every $\delta\in v_j\Cat$. Thus, $\beta_1 (a_1\cdot \delta)\leq  \beta_1\gamma (a_2\cdot \delta)$ for every $\delta\in v_j\Cat$. Then, by left cancellation, $a_1\cdot \delta \leq  \gamma (a_2\cdot \delta)$ for every $\delta\in v_j\Cat$. But this forces $\gamma=v_k$, and hence $\beta_1=\beta_2$.
		
		Now, since $\beta_1=\beta_2$, we have that $a_1=a_2$, and therefore the set $F=\{v_j\}$ satisfies the conditions of  Proposition \ref{ZP_TP}(3). So, $\Grpd_{tight}(\Cat \Join^\varphi \Grpd)$  is topologically free.
		
		To check minimality take $\beta_1,\beta_2\in \Cat$ and let $F=\{r(\beta_2)\}$. Now, let $g\in \Grpd$ be the unique element such that $s(g)=r(\beta_2)$ and $r(g)=s(\beta_1)$. Then, we have that $s(\beta_1) \Cat (g\cdot r(\beta_2))=s(\beta_1) \Cat s(\beta_1\neq \emptyset$. Thus, condition $(3)$ of Proposition \ref{ZP_Min} is satisfied, and hence $\Grpd_{tight}(\Cat \Join^\varphi \Grpd)$	is minimal.	
		
		 Now we will show that $\Grpd_{tight}(\Cat \Join^\varphi \Grpd)$ is an amenable groupoid. 
		First we will check that $\Cat^{-1} \Join^\varphi \Grpd$ is an amenable groupoid. Indeed, 
		$$\Cat^{-1} \Join^\varphi \Grpd=\{(\gamma_i^n, g)\in \Cat \Join^\varphi \Grpd: i,n\in \ZZ, r(g)=v_i\}\,,$$
		with 
		$$(\gamma_i^n, g)\cdot (\gamma_j^m, h)=(\gamma_i^{n+m}, gh)\,,$$
		for $n,m\in \ZZ$, $g,h\in \Grpd$ with $r(g)=v_i$, $s(g)=v_j$ and $r(h)=v_j$. Then, we can define the homomorphism $c:\Cat^{-1} \Join^\varphi \Grpd \to \ZZ$ by $c(\gamma_i^n, g)=n$, with $\ker\, c$ isomorphic to $\Grpd$. Therefore  $\Cat^{-1} \Join^\varphi \Grpd$ is amenable. 
		
Now, by Proposition \ref{amenability}, $\Grpd_{tight}(\Cat \Join^\varphi \Grpd)$ is amenable if and only if $\Grpd_{tight}(\Cat)$ is amenable. To determine amenability of $\Grpd_{tight}(\Cat)$, first observe that $(\Cat,\dmap)$ is a Levi category \cite[Theorem 3.3]{LV2} and, since $\Cat$ is left cancellative, we have that $\Cat$ is the Zappa-Sz\'ep product of the free category $B^*$ generated by a transversal of generators of maximal right principal ideals and the groupoid $\Cat^{-1}$ \cite[Theorem 5.10]{LW}. Thus, $\Cat\cong B^*\rtimes \Cat^{-1}$. But then $(B^*,\dmap)$ has the UFP and $\Cat^{-1}\cong \bigsqcup_{i\in \ZZ}\ZZ$ is an amenable groupoid. Moreover, since $\Cat$ is right cancellative, we have that $(\Cat , \textbf{d}, \Grpd, \varphi)$ is a pseudo-free category system by Proposition \ref{ZP_RC}.  Thus, $\Grpd_{tight}(\Cat)$ is amenable by Corollary \ref{corol_amenability}. 		
 \end{example}

\end{document}